\documentclass[12pt,a4paper]{amsart}
\usepackage{preamble}

\title
[Divergence formulas for perturbations of transfer operators] 
{
Recursive divergence formulas for perturbing unstable transfer operators and physical measures
}

\begin{document}

\begin{abstract}

We show that the derivative of the (measure) transfer operator with respect to the parameter of the map is a divergence. Then, for physical measures of discrete-time hyperbolic chaotic systems, we derive an equivariant divergence formula for the unstable perturbation of transfer operators along unstable manifolds. This formula and hence the linear response, the parameter-derivative of physical measures, can be sampled by recursively computing only $2u$ many vectors on one orbit, where $u$ is the unstable dimension. The numerical implementation of this formula in \cite{far} is neither cursed by dimensionality nor the sensitive dependence on initial conditions.

\end{abstract}

\maketitle

\section{Introduction}

\subsection{Literature review}
\hfill\vspace{0.1in}
\label{s:mao}

The transfer operator, also known as the Ruelle-Perron-Frobenius operator, describes how the density of a measure is evolved by a map, and is frequently used to study the behavior of dynamical systems.
The transfer operator was historically used for expanding maps because it makes the density smoother. 
The anisotropic Banach space of Gouëzel, Liverani, and Baladi extends the operator theory to hyperbolic maps, which has both expanding and contracting directions \cite{Gouezel2006,Gouezel2008,Baladi2017}.
The physical measure, or the SRB measure, has the eigenvalue 1, so it encodes the long-time statistics of the system, and is typically singular with respect to the Lebesgue measure \cite{srbmap,srbflow,Pacifico2022}.

The derivative of the transfer operator with respect to system parameters is useful in several settings, especially in linear response, which is the derivative of the physical measure \cite{Ruelle_diff_maps,Dolgopyat2004,Baladi2007,Gallavotti1996}.
The operator formula for the linear response (see \cref{s:two_versions})
is particularly attractive in numerical computations because it is not affected by the exponential growth of unstable vectors.
Since the physical measure is typically singular, researchers need to use finite-elements to approximate and mollify the measure.

However, When the phase space is high-dimensional, the cost of approximating a measure by finite-elements
is cursed by dimensionality (see \cref{s:costEst} for a cost estimation).
Given a $C^2$ objective function $\Phi$, the more efficient way to sample the physical measure $\rho$ of $f$ is to `sample by an orbit', that is
\[ \begin{split}
  \rho(\Phi) = \int \Phi  \,
  d\rho(x)
  \approx \frac 1K \sum_{k=1}^K \Phi(x_k),
  \quad \textnormal{where} \quad 
  x_{k+1} = f(x_k).
\end{split} \]
This approach's main cost is using a recursive relation $f$ to compute an orbit $\{x_i\}_{k=1}^K$, where each $x_i$ is essentially an $M$-dimensional vector, $M$ being the dimension of the system.

It is natural to ask whether the derivative of the transfer operator and hence the linear response can also be sampled by an orbit, that is, by a formula which involves computing several vectors recursively along a single orbit.
This is impossible for the entire derivative operator, which is typically singular for physical measures, and is not pointwise defined.
Similarly, the two most well-known linear response formulas, the ensemble formula and the operator formula, involve distributional derivatives,
which are not well-defined pointwise.

However, we typically only need the transfer operator to handle the unstable perturbations, which turns out to be well-defined pointwise.
There were pioneering works concerning pointwise formulas of the unstable part of the linear response, though they were not very clearly related to the perturbation of 
transfer operators within unstable manifolds,
which acts on conditional measures.
Ruelle mentioned how to derive a pointwise defined formula for the unstable divergence, but no explicit formulas were given \cite[lemma 2]{Ruelle_diff_maps_erratum}.
Gouëzel and Liverani gave an explicit pointwise formula via a cohomologous potential function, but the differentiation is in the stable subspace, which typically has very high dimension \cite[proposition 8.1]{Gouezel2008}.
No recursive formulas were given, and even the potential formulas are likely to involve the evolution of a lot of vectors.
It is also difficult to obtain coordinate-independent formulas.

In this paper, we derive a new divergence formula for the unstable perturbation of unstable transfer operators on physical measures. 
We then give a new interpretation of the unstable part of the linear response by such unstable transfer operators.
We think this interpretation has a more direct physical meaning compared to previous works on linear responses, which typically involve distributional derivatives or require moving to the sequence space.
Also, our formula is coordinate independent: 
it only requires a basis of the unstable subspace, but does not specify the individual vectors of a basis.
More importantly, our formulas are recursive: to evaluate the formula, we only need to track the evolution of $2u$ vectors or covectors.
Here $u$ is the dimension of the unstable subspace; we also use superscript $u$ to denote quantities related to the unstable subspace.
This number of recursive relations is likely to be minimal.

Our work bridges two previously competing approaches for computing the linear response, the orbit/ensemble approach and the measure/operator approach.
That is, we should add up the orbit change caused by stable/shadowing perturbations and the measure change caused by unstable perturbations; both changes are sampled by an orbit.
Our work may be viewed as the generalization of the well-known MCMC (Markov Chain Monte Carlo) method for sampling derivatives of transfer operators and physical measures \cite{Metropolis1953,Chorin2009}.

We compare the current paper with our other papers on related topics.
The current paper is inspired by our recent numerical algorithm, the fast (forward) response algorithm, for computing the linear response on a sample orbit \cite{fr}.
The unstable part of the fast forward response algorithm runs forward on an orbit, and the cost is linear to the number of parameters and observables.
In comparison, the current paper is the adjoint theory of sampling linear response on an orbit.
The numerical implementation of the main results in this paper is in \cite{far}.
The cost of the adjoint method for the unstable part of the linear response is independent of the number of parameters $\gamma$ and observables $\Phi$.
In other words, if we have several $X$'s, where each $X$ gives the perturbation of $f$ corresponding to changing each $\gamma$, then the cost to compute the linear responses of all $X$'s is almost the same as that of only one $X$.
Moreover, this paper generalizes the previous results to derivative operators using a new and intuitive proof by transfer operators, which is useful for the study of transient perturbations.
Finally, the papers \cite{Ni_NILSS_JCP,Ni_nilsas,Ni_asl} concerns mainly about shadowing part of the linear response, whereas the current paper concerns the unstable part, though it also uses shadowing as a tool.

\subsection{Main results}
\hfill\vspace{0.1in}
\label{s:mainrs}

As a warm up, in \cref{s:formula}, we start with the easier case where we are given a measure with a smooth density function.
Lemma~\ref{t:div} gives a divergence formula for $\delta \tL$, which is the derivative of the transfer operator $\tL$ with respect to the parameter $\gamma$ at $\gamma=0$.
Here $\tL$ is the transfer operator of $\tf$, which is a one-parameter family of maps parameterized by $\gamma$, and $\gamma\mapsto\tf$ is $C^1$ from $\R$ to $C^3$ maps on the background manifold $\cM$.
This is just the mass continuity equation on Riemannian manifolds \cite{Lamb1924}.
It can be proved simply by an integration by parts, but we shall also give a pointwise proof, which can be generalized to the more complicated scenario later on.

\begin{restatable}[mass continuity equation]{lemma}{goldbach}
\label{t:div}
For a measure with fixed smooth density $h$, and any $\tf$ such that $\delta \tf :=\left. \pp{\tf}\gamma \right| _{\gamma=0} = X$, then where $h>0$ we have
\[ \begin{split}
- \frac{\delta \tL h} h = \frac{\diverg ( h X)} h =:\div_h X.  
\end{split} \]
\end{restatable}

In the simple case above, the density $h$ is a priori known, and it is not related to a dynamical system.
But it prepares us for the more interesting case, when the measure is a physical measures of a hyperbolic attractor.
Now physical measures are defined by typical orbits of the underlying dynamical system.
Hence, it is natural to ask if the perturbation of the physical measure can be also expressed by recursive relations, on the same orbit we sample the physical measure.

\Cref{s:fastformula} proves our main result, \cref{t:vdiv}.
Let $\tL^u$ be the transfer operator of $\xi\tf$, which is a map such that $\delta(\xi\tf)=X^u$ (see \cref{s:notations} for detailed definitions), then 
\cref{t:vdiv} is a new formula for $\delta \tL^u$ on the conditional density $\sigma$ on unstable manifolds.
Here $X^u$ is the unstable part of $X$,
$\sigma$ is the density of the conditional measure of the physical measure $\rho$ for a local unstable foliation. Notice that $\sigma$ and $\delta \tL^u \sigma$ may differ by a constant factor, depending on the choice of the neighborhood of the  local foliation, so they can only be defined locally;
but the ratio $\frac{\delta \tL^u \sigma}{\sigma}$ does not depend on that choice, since the constant factors are cancelled, so it is globally well-defined.

The main significance is that this formula is defined pointwise and it involves only $2u$ recursive relations on an orbit, where $u$ is the unstable dimension.
This number should be close to the fewest possible, since we need at least $u$ modes to capture the most significant perturbative behaviors of a chaotic system, that is, there are $u$ many unstable directions.

\begin{restatable}[equivariant divergence formula]{theorem}{mianyang}
\label{t:vdiv}
Let $\sigma$ be the density of the conditional measure of $\rho$, which is the physical measure on a mixing axiom A attractor of a $C^3$ diffeomorphism $f$,
denote $\delta \tf :=\left. \pp{\tf}\gamma \right| _{\gamma=0} = X$,
then
\[ \begin{split}
  \div^u_\sigma X^u
  := - \frac{\delta \tL^u\sigma}{\sigma}
  = \div^v X + (\cS(\div^vf_*)) X.
\end{split} \]
\end{restatable}

Here $\div^v$ is the derivative tensor contracted by the unstable hypercube and its co-hypercube in the adjoint unstable subspace, so $\div^v f_*$ is a covector (see \cref{s:notations}).
$\cS$ is the adjoint shadowing operator, that is, $\omega := \cS(\div^vf_*) $ is the only bounded covector field such that $\omega = f^*\omega + \div^v f_*$.
Note that $X^u$ is not differentiable, so $\div^u_\sigma X^u$, the unstable submanifold divergence under conditional measure $\sigma$, is defined via the equivalence in the smooth situation in theorem~\ref{t:div}.
Hence, here we can not directly use \cref{t:div}, since it involves exploding intermediate quantities.

Section~\ref{s:divu} gives a new interpretation of the unstable part of the linear response by $\delta\tL^u$, and shows how to use the equivariant divergence formula to compute the linear response recursively on an orbit.
We do \textit{not} prove the linear response, the focus is to sample on an orbit.
In high-dimensional phase spaces, sampling by an orbit is much more efficient than finite element methods, whose cost is estimated in appendix~\ref{s:costEst} on a simple example.
More specifically, using the following formula, the linear response is expressed by recursively computing about $2u$ many $M$-dimensional vectors on an orbit.
Here $M$ is the dimension of the system.
Appendix~\ref{a:douyin} gives another proof of this formula using the so-called fast forward formula from a previous paper \cite{fr}.

Let $\rho$ and $\tilde \rho$ denote the SRB measure of $f$ and $\tf\circ f$, let $\Phi:\cM\rightarrow \R$ be a $C^2$ observable function.
First recall that the linear response has the expression \cite{Ruelle_diff_maps}
\[ \begin{split}
  \delta \tilde \rho(\Phi) 
  = \sum_{n\ge0} \rho(f^{n}_{*} X_{-n} (\Phi))
  = S.C. + U.C. .
\end{split} \]
Here $X_{-n}(x):=X(f^{-n}x)$, $x$ being the dummy variable in the above integration, $(f^{n}_{*} X_{-n})(x)$ is a vector at $x$, and $f^{n}_{*} X_{-n} (\Phi)=f^{n}_{*} X_{-n} \cdot \grad \Phi$. 
Here $S.C.$ and $U.C.$ are the so-called shadowing and unstable contribution,
\[ \begin{split}
  S.C. 
  := \sum_{n\ge0} \rho(f^{n}_{*} X^s_{-n} (\Phi))
  - \sum_{n\le-1} \rho(f^{n}_{*} X^u_{-n} (\Phi)) ,
  \\
  U.C 
  := \sum_{n\in \Z } \rho(f^{n}_{*} X^u_{-n} (\Phi)) .
\end{split} \]
Here  $X^u$ and $X^s$ are the unstable and stable part of $X$.
We may also decompose into stable and unstable contributions, $S.C.'$ and $U.C.'$,
\[ \begin{split}
  S.C. '
  := \sum_{n\ge0} \rho(f^{n}_{*} X^s_{-n} (\Phi)),
  \quad \textnormal{} \quad 
  U.C '
  := \sum_{n\ge0 } \rho(f^{n}_{*} X^u_{-n} (\Phi)) .
\end{split} \]
The shadowing and stable contributions are very similar in terms of numerics.
\Cref{t:vdiv} gives a new formula of $U.C.$ and hence the linear response.

\begin{restatable}[fast adjoint formula for linear response]{proposition}{lanren}
\label{t:far}
The shadowing contribution and the unstable contribution of the linear response can be expressed by integrations of quantities from the unperturbed dynamics with respect to $\rho$,
\[ \begin{split}
\delta \tilde \rho(\Phi) 
  = S.C. + U.C. ,
  \quad\quad
  S.C. = \rho (\cS(d\Phi) X) , \\
  U.C. = \lim_{W\rightarrow\infty} \rho(\phi_W \frac {\delta \tL^u \sigma} {\sigma}),
  \quad \textnormal{where} \quad 
  \phi_W:= \sum_{m=-W}^W \Phi\circ f^m .
\end{split} \]
Here $\sigma$ is the density of the conditional measure of $\rho$, and $\frac {\delta \tL^u \sigma} {\sigma}$ is given by theorem~\ref{t:vdiv}.
\end{restatable}

We explain how to compute the above formula on an orbit to overcome the curse of dimensionality.
For a finite $W$, the size of the integrand is about $\sqrt W$, and we can sample the physical measure $\rho$ by an orbit.
Then we need to compute a basis of the unstable subspace.
This can be achieved by pushing forward $u$ many randomly initialized vectors while performing occasional renormalizations, on the same orbit we used to sample $\rho$.
Similarly, we can compute a basis of the adjoint unstable subspace.
With these two basis we can compute the equivariant divergence $\div^v$.
The adjoint unstable subspace is also the main data required by the nonintrusive shadowing algorithm for computing $\cS$.
Note that we only need data obtained at $\gamma=0$, and knowing the expression of $f|_{\gamma=0}$ is sufficient for obtaining these data.

The numerical implementation of our formulas, including a detailed algorithm and numerical examples, is in \cite{far}; 
the so-called fast adjoint response algorithm is very efficient in high dimensional phase spaces; it is also robust in stochastic noise and some nonuniform hyperbolicity.
Its cost is neither cursed by  the dimensionality nor the sensitive dependence on initial conditions, and the cost is almost independent of the number of parameters.

\section{Divergence formula of derivative operator}
\label{s:formula}

We first assume that the measure on which we apply the transfer operator is smooth (means $C^\infty$) and a priori known, then we give a divergence formula for the perturbation of the transfer operator.
The techniques and notations we use for the pointwise proof of this simple case shall prepare us for the proof of \cref{t:vdiv}.
The main difference with \cref{t:vdiv} is that here the measure is not given by an orbit, so we can not sample its perturbation by an orbit.

\subsection{A functional proof}
\label{s:functional}
\hfill\vspace{0.1in}

We denote the background $M$-dimensional Riemannian manifold by $\cM$.
In this paper, we use $\tilde\cdot$ to denote perturbative quantities.
We think of perturbative maps, such as $\tf$ and the $\tL$, as being smoothly parameterized by a small real number $\gamma$, whose values at $\gamma=0$ are the identity.
Let $\tf$ be the perturbation appended to $f$, which is also a smooth diffeomorphism on 
$\cM$.
More specifically, we assume the map $\gamma\mapsto \tf$ is $C^1$ from $\R$ to the family of $C^3$ diffeomorphisms on $\cM$.
The default value of $\gamma$ is zero, and $\gamma$ may vary in a neighborhood of zero in the real numbers.
For a fixed measure with a smooth density function $h$, the transfer operator $\tL$ gives the new density function after pushing forward by $\tf$.
More specifically, $\tL$ of $\tf$ is defined by the duality
\begin{equation} \begin{split} \label{e:defDual}
  \int h \cdot (\Phi \circ \tf) =: \int \tL h \cdot \Phi,
\end{split} \end{equation}
where $\Phi:\cM\rightarrow \R$ is a $C^2$ observable function with compact support.
In this paper, all integrals are taken with respect to the Lebesgue measure, except when another measure is explicitly mentioned.
Note that $\tL$ operates on the entire density function, and $\tL h(x):=(\tL h)(x)$.
We shall refer to $h$ as the `source' of $Lh$.

We are interested in how perturbations in $\gamma$ would affect $\tL$.
Define
\[ \begin{split}
  \delta(\cdot):= \left. \frac{\partial (\cdot)}  {\partial \gamma} \right\rvert_{\gamma=0}.
\end{split} \]
We emphasize that the base value of $\gamma$ is zero, and all derivatives with respect to $\gamma$ in this paper are evaluated at $\gamma=0$.
Define the perturbation vector field $X$ as
\[ \begin{split}
X:= \delta \tf.
\end{split} \]
Note that $\tilde f$ depends on $\gamma$ and $\delta$ is the partial derivative.
Since we only consider the derivatives at $\gamma=0$, hence, we can freely assume the value of $\partial \tf/\partial \gamma$ when $\gamma\neq0$, so long as it is smooth and its value at $\gamma=0$ is $X$.
Without loss of generality, we may assume that $\tf$ is the flow of $X$.
If so, and regarding $\gamma$ as `time', then theorem~\ref{t:div} is exactly the mass continuity equation on Riemannian manifolds.

We define $\diverg_{h}$ as the divergence under the measure with density $h$,
\[ \begin{split}
  \diverg_{h} X :=\frac{ \diverg (h X)}{h} = \diverg (X) +\frac{X(h) }{h},
\end{split} \]
where $X(\cdot)$ is to differentiate a function in the direction $X$, that is, 
$X(h)=\grad h \cdot X$.
For two densities $h'$ and $h''$, if $h'= Ch''$ for a constant $C>0$, then $\diverg_{h'} = \diverg_{h''}$.
Then we prove

\goldbach*

\begin{remark*}[]
When $h=0$ in an open subset, then it typically suffices to use the fact $\delta \tL h =0$ in that open set.
\end{remark*}

\begin{proof}
Differentiate \cref{e:defDual}.
Notice that at $\gamma=0$, $\delta(\Phi \circ \tf) = \delta \tf(\Phi) = X(\Phi)$.
Then we have
\begin{equation} \begin{split} \label{e:adele}
  \int \delta \tL h \cdot \Phi
  = \int h \cdot \delta(\Phi \circ \tf)
  = \int h \cdot  X(\Phi).
\end{split} \end{equation}
We call the left hand side the operator formula, and the right side the Koopman formula for the perturbation.

Recall that $\Phi$ is compactly supported, then there is no boundary term for integration-by-parts, and we have
\[ \begin{split}
  \int \delta \tL h \cdot \Phi
  = \int h \cdot  X(\Phi)
  = - \int \div(h X) \cdot \Phi.
\end{split} \]
Since this holds for any $\Phi$, it must be $\delta \tL h = - \div(hX)$. 
%When $h$ is not compactly supported, just apply a smooth cutoff function.
\end{proof}

\subsection{A pointwise proof}
\label{s:ptwise}
\hfill\vspace{0.1in}

This section derives $\delta \tL$ using the pointwise definitions of $\tL$, which is useful later when we consider perturbations on the conditional measure on unstable manifolds.
Note that $\tL$ is equivalently defined by a pointwise expression,
\begin{equation} \begin{split} \label{e:Lh}
  \tL h(x) := \frac h { |\tf_*| }(y)
  \quad \textnormal{where} \quad 
  y:=\tf^{-1}x.
\end{split} \end{equation}
Here the point $x$ is fixed, 
whereas $y$ and $\tf$ varies according to $\gamma$, so the perturbative map $\tL$ also depends on $\gamma$. 
Here $\tf_*$ and $f_*$ are the pushforward acting on vectors, $\tf_* e := D\tf \, e$.
Later we use $f^*$ to denote the pullback acting on covectors.
$|\tf_*|$ is the Jacobian determinant, or the norm as an operator on $M$-vectors, 
\[ \begin{split}
  |\tf_*| := \frac{|\tf_* e^\cM|}{|e^\cM|},
  \quad \textnormal{where} \quad 
  e^\cM = e_1 \wcw e_M.
\end{split} \]
Here $e_i$'s are smooth 1-vector fields; $e^\cM$ is a smooth $M$-vector field, which is basically an $M$-dimensional hyper-cube field, and $|\cdot|$ is its volume.
Here $\tf_*$ is the Jacobian matrix.
Note that $|\tf_*|$ is independent of the choice of basis, and we expect this independence to hold throughout our derivation.

The volume of $M$-vectors, $|\cdot|$, is a tensor norm induced by the Riemannian metric,
\[ \begin{split}
  |e^\cM|: = \ip{e^\cM,e^\cM}^{0.5}.
\end{split} \]
For two 1-vectors, $\ip{\cdot,\cdot}$ is the typical Riemannian metric.
For simple $M$-vectors,
\[ \begin{split}
  \ip{e^\cM,r} := \det \ip{e_i, r_j},
  \quad \textnormal{where} \quad
  e = e_1\wcw e_M, \;
  r = r_1\wcw r_M, \;
  e_i, r_j\in T\cM.
\end{split} \]
When the operands are summations of simple $M$-vectors, the inner-product is the corresponding sum.

Applying $\delta$ on both sides of \cref{e:Lh}, notice that $h$ is fixed, also that $|\tf_*|=1$ when $\gamma=0$, we have
\begin{equation} \begin{split} \label{e:dL}
  \delta \tL h
  = \delta y (h)
  - h \left. \dd{}{\gamma} (|\tf_*|(y)) \right|_{\gamma=0}.
\end{split} \end{equation}
Here $\delta y=-X$, and we use it to differentiate $h$ in the coordinate variable.
Note that $\dd {}{\gamma}$ is the total derivative: $\tf$ has two direct parameters $y$ and $\gamma$, where $y$ implicitly depends on $\gamma$.
Substituting the following lemma into \cref{e:dL}, we get a pointwise proof of theorem~\ref{t:div}.

\begin{lemma} \label{l:detastar}
  $\left. \dd{}{\gamma} (|\tf_*|(y))\right|_{\gamma=0} = \diverg X $, where $X:=\delta \tf$.
\end{lemma}

\begin{proof}
  By the chain rule,
  also notice that $y|_{\gamma=0}=x$,
  the total derivative is
  \[ \begin{split}
    \left. \dd{}{\gamma} (|\tf_*|(y))\right|_{\gamma=0} = \delta |\tf_*| (x) +\delta y (\left.|\tf_*|\right|_{\gamma=0}) (x).
  \end{split} \]
  Since $|\tf_*|\equiv 1$ at $\gamma=0$, the second term is zero.
  The first term $\delta |\tf_*|$ is the partial derivative with respect to $\gamma $ while fixing $x$.
  Fix any 
  $M$-vector $e^\cM$
  at $x$, by the Leibniz rule,
  \[ \begin{split}
    \delta |\tf_*| 
    = \frac {\delta \ip{\tf_*e^\cM, \tf_*e^\cM}^{\frac 12}}{ |e^\cM|}
    = \frac 1{2|\tf_*e^\cM| |e^\cM|} \sum_{i=1}^M 2 \ip{\tf_*e_1\wcw\delta \tf_* e_i \wcw \tf_*e_M, \tf_*e^\cM}\\
    = \frac 1{|e^\cM|^2} \sum_{i=1}^M \ip{e_1\wcw\delta \tf_* e_i \wcw e_M, e^\cM}
    = \sum_{i=1}^M \eps^i \delta \tf_* e_i ,
  \end{split} \]
  where $\eps^i$ is the $i$-th covector in the dual basis of $ \{e_i \}_{ i=1 }^M$.

Because $X$ generates the flow $\tf$, we have the Lie bracket $\left[ \tf_*e_i , X\right] =0$. 
Let $\nabla_{(\cdot)}(\cdot)$ denote the Riemannian derivative, then $\nabla_X \tf_*{e_i}|_{\gamma=0}=\nabla_{e_i}X$, and
\[\begin{split}
    \delta\tf_*e_i=\nabla_X \tf_*{e_i}=\nabla_{e_i}X
\end{split}\]
Hence, we see that $\delta |\tf_*|$ is the contraction of $\nabla X$:
this is another definition of the divergence,
which is independent of the choice of the basis $\{e_i \}_{ i=1 }^M$.
\end{proof}

We may as well write the above proof using matrix determinants, which is essentially a more compact set of notations for the outer algebras we used, but is more familiar to some readers.
\footnote{This proof was suggested by a referee during the review process.}
More specifically,
\[ \begin{split}
\left.\frac{d}{d \gamma}\left|\tilde{f}_{*}\right| \circ \tilde{f}^{-1}\right|_{\gamma=0} 
& =\lim_{\gamma \rightarrow 0} \gamma^{-1}\left[\det\left(D_{\tilde{f}^{-1}} \tilde{f}\right)-1\right] \\
& =\lim_{\gamma \rightarrow 0} \gamma^{-1}\left[\det\left( \mathbbm{1} + \gamma D X+\mathcal{O}\left(\gamma^{2}\right)\right)-1\right] \\
& =\lim_{\gamma \rightarrow 0} \gamma^{-1}\left[\det\left(e^{\gamma D X+\mathcal{O}\left(\gamma^{2}\right)}\right)-1\right] \\
& =\lim_{\gamma \rightarrow 0} \gamma^{-1}\left[e^{\gamma \operatorname{Tr} D X+\mathcal{O}\left(\gamma^{2}\right)}-1\right]=\div X
\end{split} \]
The reason the determinant notation is simpler is that, in this section, we do not need the derivative of $e^\cM$: $e^\cM$ can be chosen to be essentially a constant on $\cM$.
Hence, we can use the matrix notation to completely hide away our usage of $e^\cM$.
However, the matrix notation is no longer convenient for the next section, because we will be working on submanifolds, which requires us keeping track of the derivative of the tangent space, which is non-trivial on submanifolds.

\section{Equivariant divergence formula for the unstable perturbation of transfer operator}
\label{s:fastformula}
\hfill\vspace{0.1in}

Many important measures typically live in high dimensions, such as physical measures of chaotic systems.
Efficient handling of such measures requires sampling by an orbit, since it is very expensive to approximate a high-dimensional object by finite elements, for which we give a rough cost estimation in appendix \ref{s:costEst}.
But $\delta \tL$ is not even defined pointwise for typical physical measures.
However, we only need the derivative operator to handle the unstable perturbations; the stable perturbations are typically computed by the Koopman formula on the right of \cref{e:adele}.

In this section, we derive the equivariant divergence formula of the unstable perturbation operator on the unstable manifold.
The formula is defined pointwise, moreover, it is in the form of a few recursive relations on one orbit.
As shown in figure~\ref{f:alau}, we first write the derivative operator as the derivative of the ratio between two volumes.
Then we can obtain an expansion formula, which can be summarized into a recursive formula using the adjoint shadowing lemma.

\subsection{Notations}
\hfill\vspace{0.1in}
\label{s:notations}

% subscripts
We assume that the dynamical system of the $C^3$ diffeomorphism $f$ has a mixing axiom A attractor $K$.
Denote a compact basin of the attractor by $\cV^s(K)$, which is a set containing an open neighborhood of $K$ and such that
\[ \begin{split}
  K= \cap_{n\ge0} f^n(\cV^s(K)).
\end{split} \]
There is a continuous $f_*$-equivariant splitting of the tangent vector space into stable and unstable subspaces, $V^s \bigoplus V^u$,
such that there are constants $C>0$, $0<\lambda < 1$, and
\[
  \max_{x\in K}|f_* ^{-n}|V^u(x)| ,
  |f_* ^{n}|V^s(x)| \le C\lambda ^{n} \quad \textnormal{for  } n\ge 0,
\]
where $f_*$ is the Jacobian matrix.
We still assume that the map $\gamma\mapsto \tf$ is $C^1$ from $\R$ to the family of $C^3$ diffeomorphisms on $\cM$,
and define
$X:= \delta \tf :=\left. \pp{\tf}\gamma \right|_{\gamma=0}$.
Define oblique projection operators $P^u$ and $P^s$, such that
\[ \begin{split}
  X = X^u + X^s, \quad \quad 
  X^u:=P^u X\in V^u, \quad 
  X^s:=P^s X\in V^s.
\end{split} \]
The stable and unstable manifolds, $\cV^s$ and $\cV^u$, are submanifolds tangential to the equivariant subspaces.
Superscripts of manifolds typically denote dimensions, so we also use $u$ and $s$ to denote the unstable and stable dimension, and
\[ \begin{split}
  M = s+u.
\end{split} \]
The physical measure is defined as the weak-* limit of the empirical distribution of a typical orbit.
Under our assumptions, the physical measure is also SRB, which is smooth on the unstable manifold.

We introduce some general notations to be used.
We use subscripts $i$ and $j$ to label directions, and use subscripts $m, n, k$ to label steps. 
Let $\{ e_i \}_{ i=1 }^M\subset T\cM$ be a basis vector field such that the first $u$ vectors satisfy $\span \{e_i\}_{i=1}^u=V^u$, while the other vectors satisfy $\span \{e_i\}_{i=u+1}^M = V^s$; we further require that
\[ \begin{split}
  |e|=1,
  \quad \textnormal{where} \quad 
  e:=e_1\wcw e_u.
\end{split} \]
Let $\{ \eps^i \}_{i=1  }^M$ be the dual basis covector field of $\{ e_i \}_{ i=1 }^M$, that is,
\[ \begin{split}
  \eps^ie_j =
  \begin{cases}
    1, \quad\text{if  } i=j ;
    \\
    0, \quad\text{otherwise.}
  \end{cases}
\end{split} \]
We further require that 
\[ \begin{split}
  \eps(e) = 1,
  \quad \textnormal{where} \quad 
  \eps:=\eps^1 \wcw \eps^u.
\end{split} \]
In other words, $\eps$ removes the stable component, and gives the unstable component of $u$-vectors.

The main results in our paper are coordinate-independent.
Note that $e$ belongs to the one-dimensional space  $\wedge^u V^u$.
So $e$ is the same, up to a coefficient, so long as $e_1\sim e_u$ spans $V^u$.
The case with $\eps$ is similar.
If a formula uses only the normalized $e$ and $\eps$, then it does not depend on the particular choice of $e_i$ and $\eps^i$, and we say this formula is coordinate-independent.
Indeed, most formulas in this paper involve only $e$ and $\eps$ but not individual $e_i$ and $\eps^i$.

We use $\nabla_YX$ to denote the (Riemann) derivative of the vector field $X$ along the direction of $Y$.
$\nabla_{(\cdot)} f_*$, the derivative of the Jacobian, is the Hessian 
\[ \begin{split}
  (\nabla_{Y}f_*)X 
  := \nabla_{f_*Y}(f_*X) - f_*\nabla_YX. 
\end{split} \]
This is essentially the Leibniz rule. Note that $(\nabla_{Y}f_*)X=(\nabla_{X}f_*)Y$.
Denote 
\[ \begin{split}
  \nabla_e X := \sum_{i=1}^u e_1\wcw \nabla_{e_i}X \wcw e_u,
  \quad \textnormal{} \quad 
  \nabla_{X^u} e := \sum_{i=1}^u e_1\wcw \nabla_{X^u}e_i \wcw e_u.
\end{split} \]
When $e$ is the unstable $u$-vector, it is differentiable only in the unstable direction, so in the second equation we differentiate by $X^u\in V^u$.
One of the slots of $\nabla_{(\cdot)} f_*(\cdot)$ can take a $u$-vector, in which case 
\[ \begin{split}
  (\nabla_{X}f_*)e := (\nabla_{e}f_*)X :=
  \sum_{i=1}^u f_*e_1\wcw (\nabla_{e_i}f_*) X \wcw f_* e_u,
  \\
  \nabla_{f_*e}f_*X = (\nabla_{e}f_*)X +  f_*\nabla_eX,
  \quad 
  \nabla_{f_*X}f_*e = (\nabla_{X}f_*)e +  f_*\nabla_Xe.
\end{split} \]

There are two different divergences on an unstable manifold.
The first divergence applies to a vector field within the unstable submanifold,
\[ \begin{split}
  \div^u X^u:= \ip{\nabla_e X^u, e}.
\end{split} \]
We call this the submanifold unstable divergence, or $u$-divergence.
Typically $X^u$ is not differentiable, and $\nabla_{e} X^{u}$ is a distribution rather than a function.
But $\div^u X^u$ is a Holder function: 
because \cref{t:div} shows that $\div^u_\sigma X^u$ is a transfer operator, and \cref{t:vdiv} shows that the transfer operator has an expansion formula, which is Holder.

The second kind of unstable divergence applies to vector fields not necessarily in the unstable manifold; it might be more essential for hyperbolic systems.
The two divergences coincide only if both are applied to a vector field in the unstable subspace and $V^u\perp V^s$.
Define the equivariant unstable divergence, or $v$-divergence, as
\[ \begin{split}
  \div^v X:= \eps \nabla_e X.
\end{split} \]
We define the $v$-divergence of the Jacobian matrix $f_*$,
\[ \begin{split}
  \div^v f_* :=\frac{\eps_1 \nabla_{e} f_* } {|f_* e|},
  \quad 
  (\div^v f_*) X :=\frac{\eps_1 (\nabla_{e} f_*)X } {|f_* e|},
  \quad \textnormal{where} \quad 
  \eps_1(x) := \frac {f^{*-1} (\eps(x))}{|f^{*-1} (\eps(x))|} .
\end{split} \]
By our notation, $\eps_1(x)$ is a covector at $fx$.
Note that $\div^v f_*$ is a Holder continuous covector field on the attractor.
On a given orbit, we denote
\[ \begin{split}
  e_n := \frac{f_*^ne}{|f_*^ne|},
  \quad \textnormal{} \quad 
  \eps_n := \frac{f^{*-n}\eps}{|f^{*-n}\eps|}.
\end{split} \]

Note that $\eps_1(x) = \eps(fx)$ up to an orientation.
It is convenient to assume that on the orbit we pick, the orientations are consistent, that is, in terms of vector fields,
\[ \begin{split}
  e_n =e\circ f^n,
  \quad \textnormal{} \quad 
  \eps_n = \eps\circ f^n.
\end{split} \]
This is typically true in practice, since we shall obtain unstable vectors and covectors by repeatedly pushing-forward or pulling-backward on the given orbit.
However, it is not necessary that we assume this.

\subsection{Expressing transfer operator \texorpdfstring{$\tL^u$}{} on \texorpdfstring{$\cV^u$}{unstable manifold} by holonomy map \texorpdfstring{$\xi$}{}}
\hfill\vspace{0.1in}
\label{s:xi}

% some definitions
We define the unstable perturbation on the unstable manifold as the composition of a perturbation of $\tf$ and a holonomy map $\xi$, which is a projection along stable manifolds.
As illustrated in figure~\ref{f:alau}, fix $x$,
and let $\cV^u(x)$ be the global unstable manifold through $x$ and $\cV_{r}^u(x)$ be the local unstable manifold through $x$ of size $r$ in the ambient manifold; for any $\gamma$, $\cV^{u\gamma}:=\{\tf(z):z\in\cV^u\}$ is a $u$-dimensional manifold.
For any $z\in\cV^{u\gamma}$, denote the stable manifold that goes through it by $\cV^s(z)$.
Define $\xi(z)$ as the unique intersection point of $\cV^s(z)$ and $\cV^u(x)$.
Since $\delta \tf:=\left.\pp{}{\gamma}\right|_{\gamma=0} \tf =X$ is the perturbation and $\xi$ is the projection along stable directions, if we take partial derivative with respect to $\gamma$ while fixing the base point, we have 
\[ \begin{split}
\delta(\xi\tf)
:= \left.\pp{}{\gamma}\right|_{\gamma=0}(\xi\tf)
= X^u.
\end{split} \]
Note the above equation holds only when the equation is restricted to $\cV^u$.

\begin{figure}[ht] \centering
  \includegraphics[width=1.1\textwidth,center]{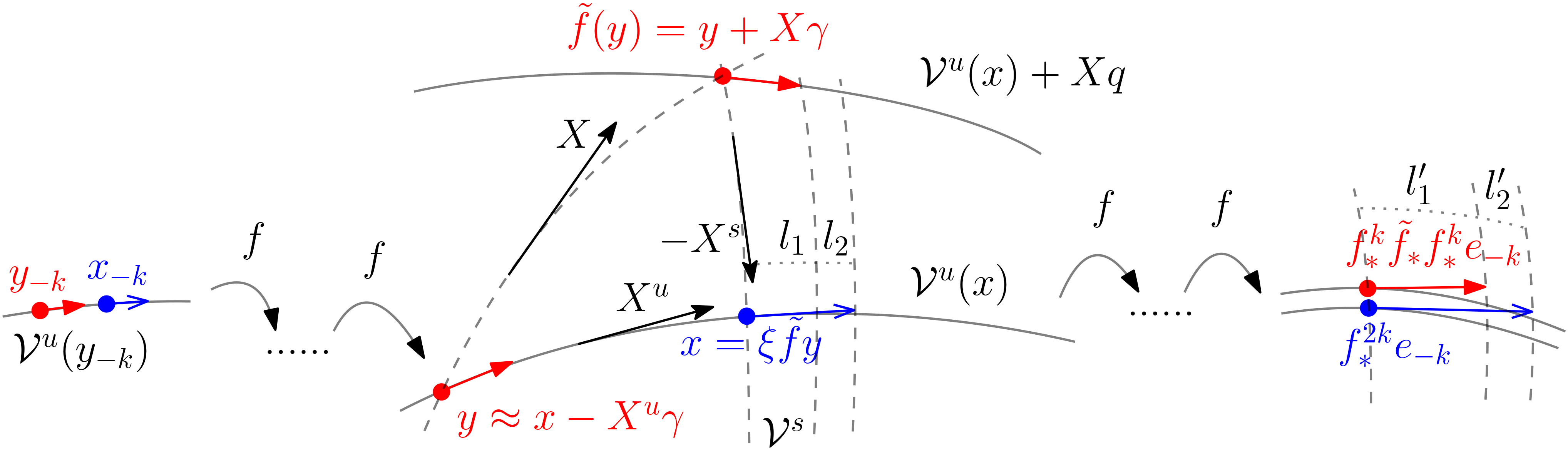}
  \caption{Definitions. 
  Here $y+X\gamma$ means to start from $y$ and flow along the direction of $X$ for a length of $\gamma$.
  Roughly speaking, 
  $\tL^u \sigma/\sigma (x) = (l_2+l_1)/l_1=(l_2'+l_1')/l_1'$,
  where $l_1, l_2$ are lengths of dotted lines.
  }
  \label{f:alau}
\end{figure}

\begin{comment}
We define $\tL^u$ as the transfer operator of $\xi\tf:\cV^u \rightarrow \cV^u$; note that $\tilde\cdot$ indicates dependence on $\gamma$.
Since $\delta(\xi \tf) = X^u$, $\delta \tL^u$ is also the perturbation by $X^u$.
Denote the density of the conditional measure at $\gamma=0$ by $\sigma$. 
Fix the fixed $x$, we want to compute $\delta \tL^u \sigma (x) $.
Let $y = (\xi\tf)^{-1} x$, so $y$ depends on $\gamma$.
By the pointwise definition of $\tL^u$,
\[ \begin{split}
  \tL^u\sigma (x)
  = \frac {\sigma} { |\xi_* \tf_{*}| } (y)
  = \frac {\sigma (y)} { |\tf_{*}(y)| \, |\xi_*(\tf y)| }.
\end{split} \]
Roughly speaking, we dissect the perturbation by $X^u$ into the perturbation of $X$ and $-X^s$.
\end{comment}

The next lemma shows that for a small range of $\gamma$, $\xi$ and $\xi\tf$ are well-defined on the entire attractor $K$; this then allows us to define locally the transfer operator $\tL^u$.
Among the many technical details below, the main facts to recall are that the unstable manifolds through any points of $K$ lie within $K$, whereas stable manifolds fill a neighborhood of $K$.

\begin{lemma} \label{l:xi}
Given any $r>0$, there is small number $\gamma_0>0$, such that, for all $|\gamma|<\gamma_0$,  for any $y\in K$, the point $x=\xi\tf y$ uniquely exists, and $|x-y|\le 0.1r$.
From now on, we always assume $|\gamma|<\gamma_0$.
\end{lemma}

\begin{proof}
Cover the compact set $K$ by a finite number of coordinate charts.
Shrink $r$ if necessary, so that for any $x\in K$, $B(x,r)$ belongs to a finite positive number of charts.
In the following paragraphs of this subsection, the angles and distances are measured in one of the charts (not the Riemannian metric).

More specifically, for all $x\in K$, the local stable manifolds $\cV_{loc}^s(x)$ and local unstable manifolds $\cV_{loc}^u(x)$ depend continuously on $x\in K$.
Since the hyperbolic set $K$ is compact, by further shrinking $r$, the sizes of local stable and unstable manifolds are uniformly larger than $r$.
By further shrinking $r$ again, we can find a positive lower bound for the angles between stable and unstable manifolds,
\begin{equation} \begin{split} \label{e:em}
  0 < \theta 
  := \inf\{\angle(V^u(z_1), V^s(z_2)): z_1\in \cV^u_r(x_1), z_2 \in \cV_r^s(x_2); x_1, x_2\in K\}.
\end{split} \end{equation}

Recall that hyperbolic attractors are isolated hyperbolic sets, so any compact basin $\cV^s(K)$ of the attractor is the union of stable manifolds through points in $K$.
More specifically, by the proof of theorem 4.26 in \cite{Wen2016} (the proof is a bit stronger than the statement of the theorem), we can see that, after further shrinking $r$ and after passing $\cV^s(K)$ to $f^n(\cV^s(K))$ for some large $n$, we have
\begin{equation} \begin{split} \label{e:ding}
\cV^s(K) \subset \cup_{z\in K} \cV^s_{0.1r}(z).
\end{split} \end{equation}

We constrain the size of our perturbation, or equivalently, constrain the range of $\gamma$, so that,
first, $\tf (K)\in \cV^s(K)$, where $\cV^s(K)$ is the basin of the attractor;
second,
\begin{equation} \begin{split} \label{e:alr}
\sup \{|\tf(y)-y|: y\in K\}< 0.1 r \, {\sin\theta}.
\end{split} \end{equation}
Since $\tf(y)\in \cV^s(K)$, by~\cref{e:ding},
there is a stable manifold going through $\tf(y)$, centered at some $z\in K$,
and $\tf(y)-z\le 0.1r$.
Since the stable manifold centered at $z$ has size larger than $r$, the stable manifold centered at $\tf(y)$ has size larger than $0.9r$.
By~\cref{e:alr} and~\cref{e:em}, 
$\cV^s_{0.9r}(\tf(y))$ and $\cV^u_r(y)$ are two transverse manifolds whose centers are close to each other,
so they intersect at a unique point $x=\xi\tf y$, and we have $|x-y|\le 0.1r$.
\end{proof}

We define $\tL^u$ as the local transfer operator of $\xi\tf:\cV^u \rightarrow \cV^u$; note that $\tilde\cdot$ indicates dependence on $\gamma$.
Let $r$ be the uniform size of local unstable manifolds; for each $x\in K$,
let $P:=(\xi\tf)^{-1}\cV^u_{0.1r}(x)$,
we define $\tL^{u}$ as the transfer operator from $C^0(P)$ function space to $C^0(\mathcal{V}_{0.1r}^{u}(x))$.
Since $\delta(\xi \tf) = X^u$, $\delta \tL^u$ is the perturbation by $X^u$.
Let $y = (\xi\tf)^{-1} x$, then the pointwise definition of $\tL^u$ on any density function $\sigma$ is
\[ \begin{split}
  \tL^u\sigma (x)
  := \frac {\sigma} { |\xi_* \tf_{*}| } (y)
  = \frac {\sigma (y)} { |\tf_{*}(y)| \, |\xi_*(\tf y)| }.
\end{split} \]
Here the last expression, roughly speaking, dissects the perturbation by $X^u$ into the perturbation by $X$ and $-X^s$.

Then we define $\frac{\tL^u\sigma}{\sigma} (x)$, where $\sigma$ is the conditional density.
Let $r$ be such that $B(x,r)\cap K$ is folicated by unstable leaves at $\gamma=0$, and $\sigma$ is the density of the conditional measure of the physical measure $\rho$.
In particular, the domain of $\sigma$ includes $\mathcal{V}_{r}^{u}(x)$, and $P:=(\xi\tf)^{-1}\cV^u_{0.1r}(x)\subset \cV^u_r(x) \subset K\cap B(x,r)$, so we can define $\frac{\tL^u\sigma}{\sigma} (x)$ on the smaller leaf $\cV^u_{0.1r}(x)$ for the particular $r$.
Moreover, notice that both $\sigma|_{\cV_{0.1r}^{u}(x)}$ and $\sigma|_P$, the source of $\tL^u \sigma|_{\cV_{0.1r}^{u}(x)}$, are restrictions of the same $\sigma$ from the same larger leaf $\cV_{r}^{u}(x)$ of the same foliation.
Hence, we can expect that, in $\frac{\tL^u\sigma}{\sigma} (x)$, the factor due to the selection of $B(x,r)$ would cancel.
Indeed, as we shall see by the expression in \cref{l:volratio} and \cref{t:vdiv}, $\frac{\tL^u\sigma}{\sigma}(x)$ and $\frac{\delta \tL^u\sigma}{\sigma}(x)$ do not involve $\sigma$ and do not depend on the selection of $B(x,r)$, and $\frac{\delta \tL^u\sigma}{\sigma}$ is a continuous function on $K$.

An additional technical subtlety is that, we want to say $\xi$ is a holonomy map, which is defined only on the hyperbolic set $K$ \cite[section 4.3]{barreirapesin}.
This follows from that the unstable manifold through any $x\in K$ lies in $K$.
As a result, we can use the standard absolute continuity lemma of holonomy maps to resolve the change of the conditional SRB measure on all unstable manifolds.

\subsection{One volume ratio for the entire unstable perturbation operator}
\hfill\vspace{0.1in}
\label{s:onevol}

\begin{lemma} [A volume ratio] \label{l:volratio}
Let $e_{-k}$ be the unit $u$-vector field on $\cV^u(x_{-k})$, where $x_{-k}:=f^{-k}x$.
Denote $y:=(\xi \tf)^{-1}x$ and $y_{-k}:=f^{-k}y$.
Then
\[ \begin{split}
  \frac{\tL^u\sigma}{\sigma} (x) 
  = \lim_{k\rightarrow\infty} \frac{|f_*^{2k} e_{-k}(x_{-k})|}{|f_*^k\tf_*f_*^ke_{-k}(y_{-k})|}
  .
\end{split}\]
\end{lemma}

\begin{proof}
First, we find an expression for $\sigma$ by considering how the Lebesgue measure on $\cV^u(x_{-k})$ is evolved.
The mass contained in the cube $e_{-k}$ is preserved via pushforwards, but the volume increased to $f_*^ke_{-k}$.
Hence, 
for $y\in \cV^u(x)$, the density $\sigma$ satisfies
\[ \begin{split}
\frac{\sigma(x)}{\sigma(y)} 
=
\lim _{k \rightarrow \infty} 
\frac
{\left|f_{*}^{k} e_{-k}(f^{-k} y) \right|}
{\left|f_{*}^{k} e_{-k}(f^{-k} x) \right|}.
\end{split} \]
This expression was stated for example in \cite[proposition 1]{Ruelle_diff_maps_erratum} using unstable Jacobians;
note that the conditional measure is determined up to a constant coefficient.
Hence,
\[ \begin{split}
  \frac{\tL^u\sigma}{\sigma} (x) 
  =  \lim_{k\rightarrow\infty}  \frac{|f_*^ke_{-k}(x_{-k})|}{|f_*^ke_{-k}(y_{-k})| \, |\tf_{*}(y)| \, |\xi_*(\tf y)|}
  .
\end{split} \]
Here $|\tf_{*}(y)| := \frac{|\tf_* e(y)|}{|e(y)|}$, 
$|\xi_{*}(\tf y)| := \frac{|\xi_*\tf_* e(y)|}{|\tf_*e(y)|}$,
where $\tf_* e(y)$ is a vector at $\tf y$, $\xi_*\tf_* e(y)$ is a vector at $x$.

By a corollary of the absolute continuity of the holonomy map \cite[theorem 4.4.1]{barreirapesin} (we provide an intuition for this corollary after this proof),
\begin{equation} \begin{split} \label{e:holo}
\frac {1} {|\xi_*(\tf y)|} 
:= \frac {|\tf_* e(y)|} {|\xi_*\tf_* e(y)|} 
= \prod_{n=0}^\infty 
\frac {|f_*^{n}\tf_* e(y)|} { |f_*^{n+1} \tf_* e(y)|}
\frac {|f^{n+1}_* \xi_*\tf_* e(y) |} {|f_*^{n}\xi_* \tf_* e(y)|} 
\\
= \lim_{k\rightarrow\infty}
\frac {|\tf_* e(y)|} { |f_*^{k} \tf_* e(y)|}
\frac {|f^{k}_* \xi_*\tf_* e(y) |} {|\xi_* \tf_* e(y)|} .
\end{split} \end{equation}
By substitution and cancellation,
\[ \begin{split}
  \frac{\tL^u\sigma}{\sigma} (x) 
  =  \lim_{k\rightarrow\infty}
  \frac{|f_*^ke_{-k}(x_{-k})|}{|f_*^ke_{-k}(y_{-k})| } 
  \frac{|e(y)|}{|f_*^k \tf_*e(y)|} 
  \frac{|f_*^{k}\xi_*\tf_* e(y)|}{|\xi_*\tf_* e(y)|} 
\end{split} \]

Both $f_*^k e_{-k}(x_{-k})$ and $\xi_*\tf_*e(y)$ are in the one-dimensional subspace $\wedge^u V^u(x)$, so the growth rate of their volumes are the same when pushing forward by $f_*$, hence
\[ \begin{split}
  \frac {|f^{k}_* \xi_*\tf_* e(y) |} {|\xi_* \tf_* e(y)|}
  = \frac {|f^{2k}_* e_{-k}(x_{-k}) |} {| f_*^k e_{-k}(x_{-k})|}.
\end{split} \]
Similarly, 
\[ \begin{split}
  \frac{|e(y)|}{|f_*^k \tf_*e(y)|} 
  = \frac{|f_*^k e_{-k}(y_{-k})|}{|f_*^{k} \tf_* f_*^{k} e_{-k}(y_{-k})|}.
\end{split} \]
Finally, by substitution and cancellation,
\[ \begin{split}
  \frac{\tL^u\sigma}{\sigma} (x) 
  =  \lim_{k\rightarrow\infty}
  \frac{|f_*^{2k} e_{-k}(x_{-k})|}{|f_*^k \tf_*f_*^k e_{-k}(y_{-k})|} .
\end{split} \]
\end{proof}

We give an intuitive explanation of \cref{e:holo}.
Because $\xi$ is projection along the stable manifolds, intuitively, for any $z$ near $K$, such as $z = \tf y$, we have
\[ \begin{split}
\lim _{k \rightarrow \infty} d(f^{k} \xi z , f^{k} z) =0 ,
\end{split} \]
where $d$ is the distance on $\cM$.
For any vector $e'$ at $z$ transverse to $V^s(z)$, such as $e'=\tf_* e(y) $, vaguely speaking, the vector $e'$ and $\xi_* e'$ collapse after pushforward many times, so we have
\[ \begin{split}
\lim _{k \rightarrow \infty} 
\frac {|f^{k}_* \xi_* e' |} {|f_*^{k} e'|} = 1.
\end{split} \]
This equation is equivalent to \cref{e:holo}, after cancellation from both sides of \cref{e:holo}.

\begin{lemma} [Expanded equivariant divergence formula] \label{l:expand}
\[ \begin{split}
  - \frac{\delta \tL^u\sigma}{\sigma} 
  = 
  \div^v X
  - \sum_{m=1}^{\infty} (\div^vf_*)_{-m} f_*^{-m} X^u
  + \sum_{n=0}^{\infty} (\div^vf_*)_{n} f_*^{n} X^s.
\end{split} \]
If we evaluate this formula at $x$, then here $X^s$ is a vector at $x$, $f_*^{n} X^s$ is a vector at $f^{n}x$, $(\div^vf_*)_{n}(x)=(\div^vf_*)(x_{n})$.
\end{lemma}

\begin{proof}
Formally differentiate the expression in \cref{l:volratio},
\[ \begin{split}
  \frac{\delta \tL^u\sigma}{\sigma} (x) 
  =  \lim_{k\rightarrow\infty} - \frac{|f_*^{2k} e_{-k}(x_{-k})|}
  {|f_*^k\tf_*f_*^ke_{-k}(y_{-k})|^2} 
  \left.\dd{}{\gamma}|f_*^k\tf_*f_*^ke_{-k}(y_{-k})| \right|_{\gamma=0} \\
\end{split} \]
At $\gamma=0$, we have $x=y$, $\tf$ is identity, so
\[ \begin{split}
  - \frac{\delta \tL^u\sigma}{\sigma} (x) 
  = \lim_{k\rightarrow\infty} \frac 1 {|f_*^{2k} e_{-k}(x_{-k})|}
  \left.\dd{}{\gamma} |f_*^k\tf_*f_*^ke_{-k}(y_{-k})|\right|_{\gamma=0}
\end{split} \]
Here
\[ \begin{split}
\left.\dd{}{\gamma} |f_*^k\tf_*f_*^ke_{-k}(y_{-k})|\right|_{\gamma=0}
  =\left. \dd{}{\gamma} \ip{f_*^k\tf_*f_*^ke_{-k}(y_{-k}),f_*^k\tf_*f_*^ke_{-k}(y_{-k})}^{\frac 12}\right|_{\gamma=0}
  \\
  = \frac12 \ip{f_*^{2k} e_{-k},f_*^{2k} e_{-k}}^{-\frac 12} 2 \ip{ \left.\dd{}{\gamma} f_*^k\tf_*f_*^ke_{-k}(y_{-k})\right|_{\gamma=0},f_*^{2k} e_{-k}}
  \\
  = \frac 1 {|f_*^{2k} e_{-k}|}
  \ip{\left.\dd{}{\gamma} f_*^k\tf_*f_*^ke_{-k}(y_{-k})\right|_{\gamma=0},f_*^{2k} e_{-k} }.
\end{split} \]
Here the second equality demands that $\dd {}{\gamma}$, applied on vectors, is the Riemannian derivative.
Moreover, we emphasize that $\dd {}{\gamma}$ is also the total derivative:
$\tf$ has two direct parameters $y$ and $\gamma$; $f$ has only one variable $y$, and $y$ depends on $\gamma$.
Summarizing, we have,
\[ \begin{split}
  - \frac{\delta \tL^u\sigma}{\sigma} (x) 
  = \lim_{k\rightarrow\infty} \frac
  {\ip{ \left.\dd{}{\gamma} f_*^k\tf_*f_*^ke_{-k}(y_{-k})\right|_{\gamma=0},f_*^{2k} e_{-k} }}
  {|f_*^{2k} e_{-k}|^2},\\
  \quad \textnormal{where} \quad 
  \dd{}{\gamma}  f_*^k\tf_*f_*^ke_{-k} = \sum_{i=1}^u  f_*^k\tf_*f_*^ke_{-k,1}\wcw \dd{}{\gamma}  f_*^k\tf_*f_*^ke_{-k,i} \wcw  f_*^k\tf_*f_*^ke_{-k,u}.
\end{split} \]

Recursively apply the Leibniz rule, note that $\tf_*=I_d$ when $\gamma=0$, we get
\[ \begin{split}
  \dd{}{\gamma} f_*^k\tf_*f_*^ke_{-k}(y_{-k})
  = 
  f_*^{k}(\dd{}{\gamma} \tf_*) f_*^{k} e_{-k}
  +
  f_*^{2k} \nabla_{-f_*^{-k}X_u}e_{-k}\\
  + 
  \sum_{n=0}^{k-1}
  f_*^{2k-n-1}(\nabla_{-f_*^{n-k} X^u} f_*) f_*^{n} e_{-k}
  + 
  f_*^{k-n-1}(\nabla_{f_*^n X^s} f_*) f_*^{n+k} e_{-k}.
\end{split} \]
Then we substitute into the previous equation to get
\begin{equation} \begin{split} \label{e:lanli}
  - \frac{\delta \tL^u\sigma}{\sigma} (x) 
  = \lim_{k\rightarrow\infty} 
  \frac{\ip{f_*^{k} \left(\left.\dd{}{\gamma} \tf_*\right|_{\gamma=0}\right) f_*^{k} e_{-k},f_*^{2k} e_{-k} }}
  {|f_*^{2k} e_{-k}|^2}
  +
  \frac{\ip{f_*^{2k} \nabla_{-f_*^{-k}X^u}e_{-k},f_*^{2k} e_{-k} }}
  {|f_*^{2k} e_{-k}|^2}
  \\
  + 
  \sum_{n=0}^{k-1}
  \frac{\ip{f_*^{2k-n-1}(\nabla_{-f_*^{n-k} X^u} f_*) f_*^{n} e_{-k},f_*^{2k} e_{-k} }}
  {|f_*^{2k} e_{-k}|^2}
  + 
  \frac{\ip{f_*^{k-n-1}(\nabla_{f_*^n X^s} f_*) f_*^{n+k} e_{-k},f_*^{2k} e_{-k} }}
  {|f_*^{2k} e_{-k}|^2} 
  .
\end{split} \end{equation}
The convergence as $k\rightarrow\infty$ is uniform for a small range of  $|\gamma|$, justifying the formal differentiation. 
Then we shall simplify each term in \cref{e:lanli} to prove the lemma.

The second term on the right of this equation is zero, since 
\[ \begin{split}
  \lim_{k\rightarrow\infty} f_*^{-k} X^u = 0.
\end{split} \]

Then we consider the first term in~\cref{e:lanli}.
\[\begin{split}
    \left.\left( \dd{}{\gamma}\tf \right)e\right|_{\gamma=0}
    :=\left.\dd{}{\gamma}(\tf_*e)-\tf_*\dd{}{\gamma}e\right|_{\gamma=0}\\
    =\left.\nabla_{\pp{}{\gamma}}(\tf_*e)\right|_{\gamma=0}
    +\nabla_{\delta y}\left(\left.\tf_*e\right|_{\gamma=0}\right)
    -\nabla_{\delta y}e
\end{split}\]
Since $\tf_*e|_{\gamma=0}=e$, we have $\nabla_{\delta y}(\tf_* e|_{\gamma=0} )=\nabla_{\delta y}e$, so the last two terms cancel each other.
Since $\tf$ is the flow of $X$, we can use the same Lie bracket statement as in \cref{l:detastar}, to get $\nabla_{\pp{}{\gamma}}(\tf_*e)=\nabla_{\tf_*e}(X)=\nabla_{e} X$ at $\gamma=0$. 
Hence,
\[ \begin{split}
\left( \left. \dd{}{\gamma} \tf_* \right|_{\gamma=0} \right) e
= \nabla_{e} X .
\end{split} \]
Then we show where $\div^vX$ in the lemma comes from. Roughly speaking, $e$ grows faster than all the other $u$-vectors, so after pushing-forward many times, $e$ becomes dominant: this is proved in theorem~\ref{l:alice} in the appendix.
Hence,
\[ \begin{split}
  \lim_{k\rightarrow\infty} 
  \frac{\ip{f_*^{k}\left(\left.\dd{}{\gamma} \tf_*\right|_{\gamma=0}\right) f_*^{k} e_{-k},f_*^{2k} e_{-k} }}
  {|f_*^{2k} e_{-k}|^2}
  =
  \lim_{k\rightarrow\infty} 
  \frac{\ip{f_*^{k}\left(\left.\dd{}{\gamma} \tf_*\right|_{\gamma=0}\right) e,f_*^{k}e }}
  {|f_*^{k} e|^2}
  \\=
  \lim_{k\rightarrow\infty} 
  \frac{\ip{f_*^{k}\nabla_e X, e_k }}
  {|f_*^{k} e|}
  =
  \eps \nabla_eX 
  =:
  \div^vX.
\end{split} \]
This is the first term in the right hand side of the lemma.

Then we consider the terms in the second last sum of~\cref{e:lanli}.
Roughly speaking, for large $k$
\[ \begin{split}
\frac{\ip{f_*^{2k-n-1}(\nabla_{-f_*^{n-k} X^u} f_*) f_*^{n} e_{-k},f_*^{2k} e_{-k} }} {|f_*^{2k} e_{-k}|^2}
%= 
%\lim_{k\rightarrow\infty} 
%\frac{\ip{f_*^{2k-n-1}(\nabla_{-f_*^{n-k} X^u} f_*) e_{n-k},f_*^{2k-n-1} e_{n-k+1} }} 
%{|f_*^{2k-n-1} e_{n-k+1}|^2
%|f_* e_{n-k}|}\\
\\
\approx
\frac{\eps_{n-k+1} (\nabla_{-f_*^{n-k} X^u} f_*) e_{n-k}} 
{|f_* e_{n-k}|}
= 
- \left(\frac{\eps_1 \nabla_{e} f_* } {|f_* e|}\right)_{n-k} f_*^{n-k} X^u \,.
\end{split} \]
To make this rigorous, use theorem~\ref{l:alice} in the appendix, which estimates the error of the above approximation, and we see that
\[ \begin{split}
\sum_{n=0}^{k-1}
\left|\frac{\ip{f_*^{2k-n-1}(\nabla_{-f_*^{n-k} X^u} f_*) f_*^{n} e_{-k},f_*^{2k} e_{-k} }} {|f_*^{2k} e_{-k}|^2}
- \frac{\eps_{n-k+1} (\nabla_{-f_*^{n-k} X^u} f_*) e_{n-k}} 
{|f_* e_{n-k}|}
\right|\\
=\sum_{n=0}^{k-1}\left| \frac{\ip{f_*^{2k-n-1}(\nabla_{-f_*^{n-k} X^u} f_*) e_{n-k},f_*^{2k-n-1} e_{n-k+1} }}
{|f_*^{2k-n-1} e_{n-k+1}|^2 |f_* e_{n-k}|}
- \eps_{n-k+1} \frac{(\nabla_{-f_*^{n-k} X^u} f_*) e_{n-k}} {|f_* e_{n-k}|}
\right|\\
\le C u \sum_{n=0}^{k-1} \lambda^{2(2k-n-1)} 
\left| \frac{(\nabla_{-f_*^{n-k} X^u} f_*) e_{n-k}} {|f_* e_{n-k}|} \right|
\le C u \sum_{n=0}^{k-1} \lambda^{2(2k-n-1)} \left|f_*^{n-k} X^u \right| \\
\le C u \sum_{n=0}^{k-1} \lambda^{2(2k-n-1)+(k-n)} \left| X \right|
\le C \lambda^{2k} |X|
.
\end{split} \]
Here the $C$'s are different in each appearance, and the last $C$ does not depend on $X$ and $k$.
Also note that the first equality employs the choice of orientation of $V^u$ on the orbit, which is implied by our notation, that is, 
\[ \begin{split}
  \frac{f_* e_{n-k}} { |f_* e_{n-k}|} = e_{n-k+1},
\end{split} \]
and it does not matter whether $e_{n}=e\circ f^{n}$ or $e_{n}=-e\circ f^{n}$ in terms of the vector fields.
Hence, the second last sum of~\cref{e:lanli} converges uniformly and absolutely, and the limit
\[ \begin{split}
\lim_{k\rightarrow\infty} 
\sum_{n=0}^{k-1}
\frac{\ip{f_*^{2k-n-1}(\nabla_{-f_*^{n-k} X^u} f_*) f_*^{n} e_{-k},f_*^{2k} e_{-k} }} {|f_*^{2k} e_{-k}|^2}
= 
\lim_{k\rightarrow\infty} 
\sum_{n=0}^{k-1}
- \left(\frac{\eps_1 \nabla_{e} f_* } {|f_* e|}\right)_{n-k} f_*^{n-k} X^u \\
= 
\lim_{k\rightarrow\infty} 
\sum_{m=-1}^{-k}
- \left(\frac{\eps_1 \nabla_{e} f_* } {|f_* e|}\right)_{m} f_*^{m} X^u 
= 
\sum_{m\le -1}
- \left(\div^v f_*\right)_{m} f_*^{m} X^u \,.
\end{split} \]
This is the first sum in the lemma.

Similarly, for the last sum in~\cref{e:lanli},
\[ \begin{split}
\sum_{n=0}^{k-1}
  \left|
  \frac{\ip{f_*^{k-n-1}(\nabla_{f_*^n X^s} f_*) f_*^{n+k} e_{-k},f_*^{2k} e_{-k} }} {|f_*^{2k} e_{-k}|^2} 
  - 
  \frac{\eps_{n+1} (\nabla_{f_*^{n} X^s} f_*) e_{n}} {|f_* e_{n}|}
  \right|
  \\ =
  \sum_{n=0}^{k-1}
  \left|
  \frac{\ip{f_*^{k-n-1}(\nabla_{f_*^n X^s} f_*) e_{n}, f_*^{k-n-1} e_{n+1} }} {|f_*^{k-n-1} e_{n+1}|^2 |f_* e_{n}| } 
  - 
  \eps_{n+1} \frac{ (\nabla_{f_*^{n} X^s} f_*) e_{n}} {|f_* e_{n}|}
  \right|
  \\ \le 
  C u \sum_{n=0}^{k-1} \lambda^{2(k-n-1)} 
  \left| \frac{(\nabla_{f_*^{n} X^s} f_*) e_{n}} {|f_* e_{n}|} \right|
  \le C u \sum_{n=0}^{k-1} \lambda^{2(k-n-1)} 
  \left| f_*^{n} X^s \right|
  \\ \le 
  C u \sum_{n=0}^{k-1} \lambda^{2(k-n-1)+n} \left| X \right|
  \le C \lambda^{k} |X|.
\end{split} \]
The convergence is absolute and uniform.
Hence, 
\[ \begin{split}
  \lim_{k\rightarrow\infty} 
  \sum_{n=0}^{k-1}
  \frac{\ip{f_*^{k-n-1}(\nabla_{f_*^n X^s} f_*) f_*^{n+k} e_{-k},f_*^{2k} e_{-k} }} {|f_*^{2k} e_{-k}|^2} 
  \\ =
  \lim_{k\rightarrow\infty} 
  \sum_{n=0}^{k-1}
  \frac{\eps_{n+1} (\nabla_{f_*^{n} X^s} f_*) e_{n}} {|f_* e_{n}|}
  =
  \sum_{n=0}^{\infty}
  (\div^v f_*)_n f_*^{n} X^s
\end{split} \]
This is the last sum in the lemma.
Summarizing, we have simplified each term in~\cref{e:lanli} and thus proved the lemma.
\end{proof}

\subsection{Recursive formula}
\hfill\vspace{0.1in}
\label{s:recurs}

The adjoint shadowing operator on covectors is equivalently defined by three characterizations \cite{Ni_asl}:
\begin{enumerate}
\item
  $\cS$ is the linear operator $\cS:\cX^{*\alpha}(K) \rightarrow \cX^{*\alpha} (K)$, such that
  \[ \begin{split}
    \rho(\om S(X)) = \rho(\cS(\om) X)
    \quad \textnormal{for any } X\in\cX^\alpha(K).
  \end{split} \]
Here $ \cX^{\alpha}(K)$ and $ \cX^{*\alpha}(K)$ denote the space of Holder-continuous vector and covector fields on $K$.
$S$ is the (forward) shadowing operator, that is, $v=S(X)$ is the only bounded solution of the variational equation $v\circ f = f_* v + X$.
  
\item
  $\cS(\om)$ has the expansion formula given by a `split-propagate' scheme,
  \begin{equation*}
    \cS ( \om ) := \sum_{n\ge 0} f^{*n} \cP^s \om_n
    -  \sum_{n\le -1}  f^{*n} \cP^u \om_n\,.
  \end{equation*}
\item 
  The shadowing covector $\nu=\cS(\om)$ is the unique bounded solution of the inhomogeneous adjoint equation,
  \[ \begin{split}
  \nu = f^* \nu_1 + \om, 
  \quad \textnormal{where} \quad \nu_1:=\nu \circ f.
  \end{split} \]
\end{enumerate}
Here $f^*$ is the pullback operator on covector, which is dual of $f_*$.
Here $\cP^s, \cP^u$, and $f^*$ are transposed matrices, or adjoint operators, of $P^s, P^u$, and $f_*$.
More specifically, define the adjoint projection operators, $\cP^u$ and $\cP^s$,
such that for any $w\in T_{x}\cM$, $\eta \in T^*_{x}\cM$,
\[ \begin{split}
  \eta (P^u w) = \cP^u \eta(w),
  \quad \textnormal{} \quad 
  \eta (P^s w) = \cP^s \eta(w).
\end{split} \]
We can show that $\cP^s, \cP^u$ in fact project to unstable and stable subspaces for the adjoint system.
This is the adjoint theory of the conventional shadowing lemma \cite{Bowen_shadowing,Wen2020}.
Then we can prove

\mianyang*

\begin{remark*}
(1) Note that the middle expression is the same for two local density function $\sigma$'s which are different only by a constant multiplier; the rightmost expression does not explicitly involve $\sigma$.
(2) This theorem can be proved via another approach, via the fast formula in \cite{fr}, as given in appendix~\ref{a:douyin}, but that proof is longer and less intuitive.
\end{remark*}

\begin{proof}
The first equality is due to the distributional definition of $\div^u_\sigma X^u$.

Since the unstable and stable subspaces are invariant, we have
\[ \begin{split}
  f_*^k P^u = P^uf_*^k 
  \,, \quad \textnormal{} \quad 
  f_*^k P^s = P^sf_*^k
  \,.
\end{split} \]
Hence, 
By lemma~\ref{l:expand}
\[ \begin{split}
  - \frac{\delta \tL^u\sigma}{\sigma} 
  = \div^v X
  - \sum_{m=1}^{\infty} (\div^vf_*)_{-m} f_*^{-m} P^uX
  + \sum_{n=0}^{\infty} (\div^vf_*)_{n} f_*^{n} P^sX \\
  = \div^v X
  - \sum_{m=1}^{\infty} (\div^vf_*)_{-m}P^u f_*^{-m} X
  + \sum_{n=0}^{\infty} (\div^vf_*)_{n}P^s f_*^{n} X.
\end{split} \]
By the definition of adjoint projection operators,
\[ \begin{split}
  - \frac{\delta \tL^u\sigma}{\sigma} 
  = \div^v X
  - \sum_{m=1}^{\infty} \cP^u(\div^vf_*)_{-m} f_*^{-m} X
  + \sum_{n=0}^{\infty} \cP^s(\div^vf_*)_{n} f_*^{n} X.
\end{split} \]
By the definition of the pullback operator $f^*$,
\[ \begin{split}
  - \frac{\delta \tL^u\sigma}{\sigma} 
  = \div^v X
  + \left(\sum_{n=0}^{\infty} f^{*n}\cP^s(\div^vf_*)_{n}
  - \sum_{m=1}^{\infty} f^{*-m}\cP^u(\div^vf_*)_{-m}  \right) X.
\end{split} \]
By the expansion formula of the adjoint shadowing operator, we prove the lemma.
\end{proof}

The significance of the formula in \cref{t:vdiv} is that it can be sampled by an orbit, just as how the physical measure is defined.
More specifically, this means two things
\begin{itemize}
  \item The formula is defined pointwise.
  \item The formula can be computed by recursively applying a map on a few vectors and covectors.
\end{itemize}

The formula is defined pointwise: all differentiations hit only $X$ and $f_*$:
these vector or tensor fields are at least $C^2$ on $\cM$.
Also, all intermediate quantities have bounded sup norm.
There are previous works achieving pointwise formula for the unstable part of linear responses \cite{Ruelle_diff_maps_erratum,Gouezel2008}.
However, it was not clear back then that those formulas were related to unstable transfer operators, and the formulas are not recursive.

More importantly, our formula can be sampled on an orbit via only $2u+1$ many recursive relations.
The numerical implementations, including several none-trivial tricks such as renormalizations and matrix notations, and several numerical examples are given in another paper \cite{far}.
It is inobvious, if not impossible, that previous pointwise formulas, even with extra work, can be realized by this many recursive relations.
First, $e$ can be efficiently computed via $u$-many forward recursion.
Since unstable vectors grow while stable vectors decay, we can pushforward almost any set of $u$ vectors, and their span will converge to $V^u$, while their normalized wedge product converges to $e$.
Note that this convergence is measured by the metric on the Grassmannian of $u$-dimensional subspaces.
Numerically, we need to perform occasional renormalizations when pushing forward the set of $u$ many single vectors; renormalization does not change the span, but avoids the clustering of single vectors.
This classical result is used in the algorithms for Lyapunov vectors by Ginelli and Benettin \cite{Ginelli2007_CLV,Ginelli2013_CLV,Benettin1980_LE}, although here we only need the unstable subspace instead of individual unstable Lyapunov vectors.
Similarly, $\eps$ can be efficiently computed via pulling-back $u$-many covectors, since it is the unstable subspace of the adjoint system.

The adjoint shadowing form $\nu:=\cS(\div^vf_*)$ can also be efficiently computed with one more backward recursion and an orthogonal condition at the first step.
The second characterization in the adjoint shadowing lemma states that $\nu$ is the only bounded solution of the inhomogeneous adjoint equation,
\[ \begin{split}
  \nu_n = f^*\nu_{n+1} + (\div^vf_*)_{n}.
\end{split} \]
Hence $\nu$ can be well approximated by solving the following equations for $a_1,\ldots, a_u \in\R$,
\[ \begin{split} 
  \nu  =  \nu ' + \sum_{i=1}^u \eps^i a_i \,,
  \quad  \mbox{s.t. }
  \ip{\nu_0, \eps^i_0}=0
  \quad \textnormal{for all} \quad 
  1\le i\le u.
\end{split}\]
Here $\nu'$ is a particular inhomogeneous adjoint solution.
Intuitively, the unstable modes are removed by the orthogonal projection at the first step, where the unstable adjoint modes are the most significant.
This is known as the nonintrusive (adjoint) shadowing algorithm \cite{Ni_nilsas} (also see \cite{Blonigan_2017_adjoint_NILSS,Ni_NILSS_JCP,Ruesha}).

Hence, we can compute the $v$-divergence formula on a sample orbit, with sampling error $E\sim O(1/\sqrt T)$, and the cost is
\begin{equation} \begin{split} \label{e:hua}
  S\sim O(uT) \sim O(u E^{-2}).
\end{split} \end{equation}
In particular, this is not cursed by dimensionality.
Compared with the zeroth-order finite-element method for the whole $\delta L$, whose cost is estimated in appendix~\ref{s:costEst}, the efficiency advantage is significant when the dimension is larger than 4.
The numerical implementations of our formula takes seconds to run on an $M=21$ system, which is almost out of reach for finite-element methods \cite{far}.

\section{Sampling linear responses by an orbit}
\label{s:divu}

This section uses our equivariant divergence formula to sample linear responses recursively on an orbit, which is the derivative of the physical measure with respect to the parameter of the system.
We do \textit{not} reprove linear responses, rather, the focus is to sample it by recursively applying a map to evolve vectors.
We first review the two linear response formula of physical measures.
Then we explain how to blend the two linear response formulas for physical measures.
In particular, the unstable part is given by the unstable perturbation of the unstable transfer operator, which can be sampled by an orbit, according to our \cref{t:vdiv}.

\subsection{Two formulas for linear response and their formal derivations}
\hfill\vspace{0.1in}
\label{s:two_versions}

The application that we are interested in is the linear response of physical measures.
In fact, most cases where we favor the derivative of transfer operators are when the perturbation is evolved for a long-time, as in the case of linear responses. 
Otherwise, if we are interested in the perturbation for only a few steps, we may as well use the Koopman formula, the left side of \cref{e:adele}, which can be evaluated much more easily than derivative of transfer operators.
This subsection reviews the linear response and its two more well-known formulas, the ensemble formula and the operator formula.
The arguments in subsection are all formal, and the purpose is to help readers review those formulas, whose rigorous proof is much more difficult than our presentation.

We stil use $f$ to denote the fixed base diffeomorphism, $\tf$ the perturbation, and $\xi$ the projection along stable foliations.
At $\gamma=0$, $\tf$ is identity. 
In this subsection, we assume that $\tf\circ f$ is still hyperbolic; this can be achieved by further shrinking the range of $\gamma$ prescribed in \cref{l:xi}.
Let $\th$ and $h$ denote the `density' of the physical measure $\tilde \rho$ and $\rho$.
Note that $\th$ is \textit{not} single-step perturbed densities, which is denoted by $\tL h$ instead.
More specifically,
\[ \begin{split}
  \th := \lim_{n\rightarrow\infty} (\tL L)^n \mu,
  \quad \textnormal{} \quad 
  h := \lim_{n\rightarrow\infty} L^n \mu,
\end{split} \]
where $\mu$ is any smooth density function of a measure supported on the basin of the attractor.
The convergence is in the weak-* sense; that is, for any $C^2$ function $\Phi$ over $\cM$,
\[ \begin{split}
  \int \Phi \th = \lim_{n\rightarrow\infty} \int \Phi\circ (\tf f)^n \mu.
\end{split} \]
The physical measure encodes the long-time-average statistics, and it has regularities in the unstable directions for axiom A systems.
A perturbation $\tf$ gives a new physical measure, and their linear relation was discussed by the pioneering works \cite{Gallavotti1996,Llave86,Jiang06,Bonetto06}, then justified rigorously for hyperbolic systems.
There are other attempts to compute the linear response which do not need ergodic theory, such as the gradient clipping and the reservoir computing method from machine learning \cite{clip_gradients2,Luca2022}.

One way to derive the linear response formula is to average the perturbation of individual orbits over the physical measure.
More specifically, for a smooth observable function $\Phi$,
\[ \begin{split}
  \int \Phi \th = \lim_{n\rightarrow\infty} \int \Phi\circ (\tf f)^n \mu.
\end{split} \]
Apply the chain rule recursively,
\[ \begin{split}
   \delta (\tf f)^n
  =  \sum_{m=0}^{n-1} f_*^m \delta \tf \circ f^{n-m} 
\end{split} \]
To intuitively explain the chain rule, consider an orbit starting from $x$, running for $n$ steps.
Then we add a perturbation $\tf$ to each step, and above expression gives how $x_n$ is perturbed.
Hence, the integrand can be expressed by the perturbation of an orbit, and 
\[ \begin{split}
   \delta(\Phi\circ (\tf f)^n) 
   = d\Phi \circ f^{n} \cdot \sum_{m=0}^{n-1} f_*^m  \delta \tf \circ f^{n-m}
  =  \sum_{m=0}^{n-1} \delta \tf(\Phi\circ f^m ) \circ f^{n-m} 
\end{split} \]
Here the last equality is the definition of pushforward of vectors.
We may also directly obtain the above formula from a bigger chain rule including $\Phi$, but here we take a detour to explain the relation to the perturbation of orbits.
Hence, the linear response is
\[ \begin{split}
  \delta \left(\int \Phi \th \right) 
  = \lim_{n\rightarrow\infty} \int \delta(\Phi\circ (\tf f)^n) \mu 
  = \lim_{n\rightarrow\infty} \sum_{m=0}^{n-1} \int \delta \tf(\Phi\circ f^m ) \circ f^{n-m} \mu \\
  = \lim_{n\rightarrow\infty} \sum_{m=0}^{n-1} \int X (\Phi\circ f^m ) \circ f^{n-m} \mu 
  = \sum_{m=0}^{\infty} \int X(\Phi\circ f^m ) h
  = \sum_{m=0}^{\infty}  \int f_*^m X(\Phi) h.
\end{split} \]
We call this the ensemble formula for the linear response,
because it is formally an average of orbit-wise perturbations over an ensemble of orbits.

For contracting maps, the ensemble formula converges, and we we only need one orbit to sample each attractor and its perturbation.
For hyperbolic systems, above formula was proved in \cite{Ruelle_diff_maps,Dolgopyat2004}.
This formula was numerically realized in \cite{Lea2000,eyink2004ruelle,lucarini_linear_response_climate,lucarini_linear_response_climate2}.
However, due to exponential growth of the integrand, 
it is typically unaffordable for ensemble methods to actually converge.
This issue is sometimes known as the `gradient explosion'.

The dual way to derive the linear response formula is to differentiate 
\[ \begin{split}
  \th = \lim_{n\rightarrow\infty} (\tL L)^{n}  \mu,
\end{split} \]
where $\mu$ is the density of a measure absolute continuous with respect to Lebesgue and supported on the attractor basin.
Since the transfer operator is, very vaguely speaking, a giant matrix, so we can formally apply the Leibniz rule at $\gamma=0$,
\[ \begin{split}
  \delta \th = \lim_{n\rightarrow\infty}  \sum_{m=0}^{n-1} L^{m} \delta \tL \, L^{n-m} \mu
  = \sum_{m=0}^{\infty} L^{m}  \delta \tL \lim_{n\rightarrow\infty} L^{n-m} \mu
  =\sum_{m=0}^\infty L^{m} \delta \tL h,
\end{split} \]
where we formally interchanged the limit and the differentiation, 
and used the fact that that $\lim_{n\rightarrow\infty} L^{n-m} \mu = h$.
Hence, we get the so-called operator formula for the linear response,
\[ \begin{split}
  \delta \int \Phi \th 
  =\int \Phi \delta \th
  =\sum_{m=0}^\infty \int \Phi L^m \delta \tL h
  =\sum_{m=0}^\infty \int \Phi \circ f^m \delta \tL h
  =\sum_{m=0}^\infty \rho (\Phi \circ f^m \frac {\delta \tL h} h) .
\end{split} \]
Here $\rho$ is the physical measure of $f$.
The convergence is due to decay of correlations, 
since formal integration-by-parts shows that,
\[ \begin{split}
  \rho\left(\frac {\delta \tL h} h \right)
  = \rho\left(\mathbbm{1} \frac {\delta \tL h} h \right)
  = \rho\left( X(\mathbbm{1}) \right)
  = 0.
\end{split} \]
Here $\mathbbm{1}$ is a constant function.
Since this term has zero mean, the decay of correlation could be faster than normal \cite{Gouezel04,Korepanov2016}.

Above formal arguments can be directly proved for expanding map $f$ with $C^3$ regularity, for example for tent maps on circles or cat maps on torus \cite{Baladi2007,Galatolo2023}.
Because, roughly speaking, expanding maps make the densities smoother, so $h$ and $\th$ are well-defined functions, and $\delta\tL h$ is also a function.
For expanding maps, it is not hard to imagine that we can sample $\frac {\delta \tL h} h$ on an orbit recursively; to do this, just linearly remove the stable part from our equivariant divergence formula.

For the case of hyperbolic sets, $h$ and $\delta\tL h$ are no longer functions.
They belong to the so-called anisotropic Banach space \cite{Liverani1995,Dolgopyat1998,Baladi1999,Gouezel04,Collet2004}.
However, the objects in the anisotropic Banach space are not pointwisely defined functions, so there is no hope to compute $\frac {\delta \tL h} h$ on an orbit.
Hence, to keep using the operator formula, we have to use finite elements to to generate a mollified approximation of the singular objects.
As shown in appendix~\ref{s:costEst}, the cost of the finite element method is exponential to $M$, which is too high for typical physical systems.
This issue is known as `curse by dimensionality'.

Finally, the operator formula and the ensemble formula are formally equivalent under integration-by parts.
From the ensemble formula,
\[ \begin{split}
  \sum_{m=0}^{\infty}  \int f_*^m X(\Phi) h
  = \sum_{m=0}^{\infty}  \int  X(\Phi\circ f^m) h
  = - \sum_{m=0}^{\infty} \int \Phi \circ f^m (\diverg_{h} X) h.
\end{split} \]
By theorem~\ref{t:div}, we have
\[ \begin{split}
  \sum_{m=0}^{\infty} \int \Phi \circ f^m (\diverg_{h} X) h
  = - \sum_{m=0}^{\infty} \int \Phi L^m \delta \tL h
  = - \int \Phi \delta \th.
\end{split} \]

To summarize, both the ensemble formula and the operator formula give the true derivative for hyperbolic systems, which have both expanding and contracting directions.
However, in high dimensions, we want to sample by orbits, and the ensemble formula is still suitable mainly for contracting systems, whereas the operator formula is suitable mainly for expanding systems.

\subsection{Blending two linear response formulas}
\hfill\vspace{0.1in}
\label{s:blend}

%For a physical measure $h$, though singular, we can sample the measure efficiently by an orbit.
%So it is natural to ask if we could also sample $\delta \tL h$ by an orbit.
%No, since $\delta \tL h$ is a distribution, and it is not defined at a single point on an orbit.
%But for many applications, we only need a part of $\delta \tL h$, which is well-defined at a point, and has the potential to be sampled by an orbit.

It is a natural idea to combine the two linear response formulas.
That is, to look at the orbit change for the stable or shadowing part of the linear response, and the density change for the unstable part.
This subsection derives this blended formula for the linear response, which can be sampled by an orbit.

We call such linear response formulas the blended formulas, such as the formula in \cite[proposition 8.1]{Gouezel2008}.
Another blended formula, which has a numerical implementation, is the one used by the blended response algorithm \cite{Abramov2008}.
But the blended response algorithm computed the unstable divergence by summing directional derivatives, which is not defined pointwise.
With our formula for the unstable perturbations of transfer operators, we can sample the unstable divergence and hence the entire linear response by an orbit.
\Cref{t:far} is part of the so-called fast (adjoint) response algorithm. It is numerically demonstrated on a 21 dimensional example with 20 unstable dimensions \cite{far,fr}.

In this subsection we shall assume linear responses proved by previous literature.
Let $\rho$ and $\tilde \rho$ denote the SRB measure of $f$ and $\tf\circ f$ supported on an axiom A attractor.
The map $\gamma\mapsto \tf$ is $C^1$ from $\R$ to the family of $C^3$ diffeomorphisms on $\cM$,
Let $\Phi:\cM\rightarrow \R$ be a $C^2$ observable function.
The linear response has the expression \cite{Ruelle_diff_maps,Dolgopyat2004,Jiang06}
\[ \begin{split}
  \delta \tilde \rho(\Phi) 
  = \sum_{n\ge0} \rho(f^{n}_{*}  X_{-n} (\Phi)).
\end{split} \]
Here $X_{-n}(x) = X(x_{-n})$, $x$ is the dummy variable in the integration, $f^{n}_{*} X_{-n}$ is a vector at $x$.
We do \textit{not} give a new proof of this formula, rather, we shall give a new formula which can be sampled by $2u$ recursive relations on an orbit.
The proof of our new formula starts from the above formula.
We define the shadowing and unstable contribution of the linear response as:
\[ \begin{split}
  S.C. 
  := \sum_{n\ge0} \rho(f^{n}_{*} X^s_{-n} (\Phi))
  - \sum_{n\le-1} \rho(f^{n}_{*} X^u_{-n} (\Phi)) ,
  \\
  U.C 
  := \sum_{n\in \Z } \rho(f^{n}_{*} X^u_{-n} (\Phi)) .
\end{split} \]

\lanren*

\begin{proof}
Recall that we have
\begin{equation*}
  \cS ( \om ) := \sum_{n\ge 0} f^{*n} \cP^s \om_n
  -  \sum_{n\le -1}  f^{*n} \cP^u \om_n\,
\end{equation*}
so we can calculate that
 \[\begin{split}
     \rho (\cS(d\Phi) X)=\rho((\sum_{n\geq 0}f^{*n}\cP^sd\Phi_n-\sum_{n\leq -1}f^{*n}\cP^ud\Phi_n)X) 
     \\
     =\rho(\sum_{n\geq 0}d\Phi_n(f^n_*P^sX)-\sum_{n\leq -1}d\Phi_n(f^n_*P^uX)).
\end{split}\]
Notice that $\rho$ is the invariant measure, so we have
\[\begin{split}
    \rho(d\Phi_n(f^n_*P^sX))=\rho(f^n_*X^s_{-n}(\Phi)) \\
     \rho(d\Phi_n(f^n_*P^uX))=\rho(f^n_*X^u_{-n}(\Phi))
\end{split}\]
Take it back to the formula, then we have
$$S.C.=\rho (\cS(d\Phi) X)$$

For the unstable contribution, take a Markov partition of the attractor so that each rectangle is foliated by unstable leaves. 
%is contained in a local unstable foliation. So each rectangle has a natural local unstable foliation.
Let $\sigma'$ denote the factor measure (also called the quotient measure), and $\sigma$ denote the conditional density of $\rho$ with respect to this foliation. Notice that the boundary of the Markov partition has zero physical measure, we take the following integral on the rectangles. 
\[\begin{split}
    \rho(f^n_*X^u_{-n}(\Phi))=\iint f^n_*X^u_{-n}\cdot\grad\Phi\sigma dxd\sigma'(x)
    =\iint \sigma X^u\cdot\grad(\Phi\circ f^n)dxd\sigma'(x).
\end{split}\]
Integrate-by-parts on unstable manifolds, note that the flux terms on the boundary of the partition cancel (this technique was used for example in \cite{Ruelle_diff_maps_erratum}), so 
\[\begin{split}
    \rho(f^n_*X^u_{-n}(\Phi))
    =\iint -\div_\sigma^uX^u(\Phi\circ f^n)\sigma dxd\sigma'(x)
    =\rho(\Phi\circ f^n\frac{\delta\tL^u\sigma}{\sigma}).
\end{split}\]
Take it back to the formula and we can finish the proof.
\end{proof}

\subsection{A formal but intuitive derivation of our blended formula}
\hfill\vspace{0.1in}
\label{s:axx}

The proof of the fast adjoint response formula in the previous subsection feels like taking a detour.  
It is based on linear response formulas proved by other people, whose proofs typically require moving to another space such as the sequence space, or anisotropic Banach space.
This section formally derives the unstable contribution in \cref{t:far} by transfer operators.
This approach is more direct, since we proved \cref{t:vdiv} by transfer operators.
Moreover, we only need to work within $\cM$ in this formal argument.

In this subsection we \textit{assume} that linear response exists and is truly linear, whose proof should still requires using advanced spaces.
More specifically, denote $\delta \rho (\Phi, X)$ as the linear response corresponding to a perturbation $\delta\tf=X$, then we assume that we already know
\[ \begin{split}
  \delta \tilde \rho (\Phi, X) 
  = \delta \tilde \rho (\Phi, X^s) 
  + \delta \tilde \rho (\Phi, X^u)  .
\end{split} \]

Define the `stable contribution' of the linear response as the linear response caused by the perturbation $X^s$.
Since stable vectors decay exponentially via pushforward, the argument in the first half of \cref{s:two_versions} still applies, so
\[ \begin{split}
  \delta \tilde \rho (\Phi, X^s)  = \sum_{m=0}^{\infty}  \rho (f_*^m X^s(\Phi))   .
\end{split} \]
This formula can be naturally sampled by an orbit.

The rest of this subsection derives the unstable contribution of the total linear response, which is the linear response caused by perturbing $\tilde \rho$ by $X^u$.
Recall that the physical measure is the weak-* limit of pushing forward a Lebesgue measure.
Intuitively, since the stable direction contracts the measures, the densities of the measure asymptotically approach the unstable manifolds, and eventually, the physical measure is carried by the unstable manifolds.
The unstable perturbation by $X^u=\delta (\xi\tf)$ will re-distribute the densities within each unstable manifold, but will not move densities across different unstable manifolds.
Hence, we can think of the dynamical system as time-inhomogeneous, hopping from one unstable manifold to another: this model is purely expanding.
In this model, the phase space is a family of unstable manifolds, which is preserved under the unstable perturbation. 
Hence, roughly speaking, the transfer operator version of linear response formula still applies.

More specifically, first, we fix a Markov partition for $f$, denote the local foliation of local unstable manifolds by $F^u$, and let $\sigma'$ be the quotient measure (not density) of the physical measure $\rho$ in the stable direction.
We also fix $F^u$, which is not changed by the unstable perturbation.
Since the unstable perturbation does not change the attractor, we can further fix $\sigma'$ for the rest of this subsection, even though $\sigma'$ is no longer the quotient measure of $\xi\tf f$ under $F^u$ when $\gamma\neq0$.
We shall designate $\delta \tilde \rho$ to the change of $\sigma$, the density of the conditional measure, which is no longer a conditional probability when $\gamma\neq0$.

%Let $s$ be a signed density function of the conditional measure of an arbitrary measure $\mu$, conditioned on the local foliation $F^u$  and quotient measure $\sigma'$.

After we fixed the quotient measure, i.e. the factor measure $\sigma'$ on leaves, then we can define a conditional density $s$ for each signed Radon measure $\mu$ which is absolutely continuous to the product measure $Leb\times \sigma'$ on each rectangle, where $Leb$ denote the Lebesgue measure on each leaf. Conversely, if we have a positive function $s$ which is integrable with respect to $Leb\times \sigma'$ we can get a Radon measure $\mu$.
We denote $x$ as a position on the manifold, and $y$ as the unstable leaf which contains $x$. So actually $y=y(x)$ depends on $x$, and we also set $d\sigma'(x)=d\sigma'(y(x))$ under our notions.

Let $L^u s$ be the conditional density of the pushforward measure $f_*\mu$ and $J^u_f(x)$ be the Jacobian of $f$ restricted to the unstable manifolds. 
Then we have
\begin{equation} \begin{split} \label{e:ww}
  L^u s(fx)J^u_f(x)dxd\sigma'(fx)
  = s(x)dxd\sigma'(x) .
\end{split} \end{equation} 
holds under the sense of integral, i.e. for any $A$ with positive $Leb\times \sigma'$ measure, take integral on $A$ for both side, this equation holds.
To see this, let $A$ be any set with positive $Leb\times \sigma'$ measure;  without loss of generality, we can let $A$ small, so that $A$ and $fA$ are each contained in a rectangle.
Then by changing variables,
\[ \begin{split}
    \int_{y\in F^u}\int_{x\in \cV^u(y)\bigcap A} s(x)dxd\sigma'(y)=\mu(A)
  \\ 
  = f_*\mu(fA) 
  = \int_{y\in F^u}\int_{x\in \cV^u(y)\bigcap A}L^u s(fx)d(fx)d\sigma'(fy).
  \\
  = \int_{y\in F^u}\int_{x\in \cV^u(y)\bigcap A}L^u s(fx)J^u_f(x)dxd\sigma'(fy).
\end{split} \]
Hence \cref{e:ww} holds.

On the set where $s(x)$ is positive, $\frac{L^u s(fx)}{s(x)}dxd\sigma'(fx)
= \frac{1}{J^u_f(x)}dxd\sigma'(x) $ holds under the sense of integral. 
As a result, if $s_1$ is everywhere positive (even if $s_2$ is not always non-zero), then we have\begin{equation} \begin{split} \label{e:zongz}
  \frac{L^u s_2(fx)}{L^u s_1(fx)}
  =
  \frac{s_2(x)}{s_1(x)}
\end{split} \end{equation} 
holds for $Leb\times\sigma'\ a.e.\   x\in K$.
In particular, it holds almost everywhere for the physical measure $\rho$.
We mention that all the point-wise equations later in this section should be understood in the sense of $Leb\times\sigma'\ a.e.\   x\in K$.

Let $\sigma_0$ be a conditional density of a measure independent of $\gamma$. 
Denote $\sigma_{n+1}:=\tL^uL^u\sigma_n$, where $L^u$ is the renormalized transfer operator of $f$ defined above.
Let $\tL^u$ be the transfer operator of $\xi\tf$. 
All equations below are still evaluated at $\gamma=0$.
First, by the chain rule, 
\[ \begin{split}
\delta \sigma_{n+1}
=\delta\tL^uL^u\sigma_n+\tL^uL^u\delta\sigma_n
=\delta\tL^u\sigma_{n+1}+L^u\delta\sigma_n,
\end{split} \]
where the last equality uses that $\tL^uL^u=L^u$ at $\gamma=0$.
Divide both sides by $\sigma_{n+1}=L^u\sigma_n>0$, to get
\[ \begin{split}
\frac{\delta\sigma_{n+1}}{\sigma_{n+1}}
= \frac{\delta\tL^u\sigma_{n+1}}{\sigma_{n+1}}
+\frac{L^u\delta\sigma_n}{L^u\sigma_n}.
\end{split} \]
Substitute $s_2=\delta\sigma_n$ and $s_1=\sigma_n$ into \cref{e:zongz}, we get
\[ \begin{split}
\frac{L^u\delta\sigma_n(x)}{L^u\sigma_n(x)}
=\frac{\delta\sigma_n(f^{-1}x)}{\sigma_n(f^{-1}x)}
\end{split} \]
By substitution, we have 
\[ \begin{split}
\frac{\delta\sigma_{n+1}}{\sigma_{n+1}}(x)=\frac{\delta\tL^u\sigma_{n+1}}{\sigma_{n+1}}(x)+\frac{\delta\sigma_n}{\sigma_n}(f^{-1}x).
\end{split} \]
Apply this equation repeatedly, notice that $\delta\sigma_{0}=0$, we get 
\[ \begin{split}
\frac{\delta\sigma_n}{\sigma_n}(x)=\sum_{k=1}^{n}\frac{\delta\tL^u\sigma_k}{\sigma_k}(f^{k-n}x).
\end{split} \]
By the same argument in the proof of \cref{t:vdiv}, we can see that the right side of the formula is independent of the choice of the Markov partition and is continuous across the boundaries of the sets in the partition.

\begin{comment}
    Now, let $h$ be the density of $\rho$, the physical measure of $f$.
Let $h_n=(\tL L)^n h$ where $\tL$ is the transfer operator of $\xi\tf$. 
We set $\sigma_n$ and $\sigma$ as the conditional measure of $h_n$ and $h$ under $F^u$ and $\sigma'$. 
By definition, $\sigma_{n+1}=\tL^uL^u\sigma_n$, so 
\[ \begin{split}
\delta\int\Phi h_n=\delta\int d\sigma'\int\sigma_n|_{\gamma=0}\frac{\sigma_n}{\sigma_n|_{\gamma=0}}\Phi
=\int d\sigma'\int\sigma_n|_{\gamma=0}\frac{\delta\sigma_n}{\sigma_n|_{\gamma=0}}\Phi
\\
=\int \frac{\delta\sigma_n}{\sigma_n}\Phi h_n
=\int \sum_{k=1}^{n}\frac{\delta\tL^u\sigma_k}{\sigma_k}(f^{k-n}x)\Phi h_n (x).
\end{split} \]
Notice that at $\gamma=0$, we have $h_n=h, \sigma_k = \sigma$, and $\rho$ is $f$-invariant.
Let $m=n-k$,
\[ \begin{split}
\delta\int\Phi h_n
=
\sum_{k=1}^{n} \rho\left(\frac{\delta\tL^u\sigma}{\sigma} \circ f^{k-n} \, \Phi\right)
=
\sum_{m=0}^{n-1} \rho\left(\Phi\circ f^m\frac{\delta\tL^u\sigma}{\sigma} \right).
\end{split} \]
Finally, let $n\rightarrow\infty$, we get  
\[ \begin{split}
\delta \tilde \rho (\Phi, X^u)
= \delta(\int\Phi\th)
= \sum_{m\geq 0}\rho\left(\Phi\circ f^m\frac{\delta\tL^u\sigma}{\sigma} \right).
\end{split} \]
This is equivalent to the expression in \cref{t:far}, except for that here we are using the stable/unstable decomposition instead of the shadowing/unstable decomposition of the linear response.
\end{comment}
Now, let $\rho$ be the physical measure of $f$.
Let $\rho_n=(\tL L)^n \rho$ where $\tL$ is the transfer operator of $\xi\tf$. 
We set $\sigma_n$ and $\sigma$ as the density of the conditional measures of $\rho_n$ and $\rho$ under $F^u$ and $\sigma'$. 
By definition, $\sigma_{n+1}=\tL^uL^u\sigma_n$, so we have
\[ \begin{split}
\delta\rho_n(\Phi)=\delta\int d\sigma'\int\sigma_n|_{\gamma=0}\frac{\sigma_n}{\sigma_n|_{\gamma=0}}\Phi
=\int d\sigma'\int\sigma_n|_{\gamma=0}\frac{\delta\sigma_n}{\sigma_n|_{\gamma=0}}\Phi
\\
=\rho_n\left(\frac{\delta\sigma_n}{\sigma_n}\Phi\right)
=\rho_n\left(\sum_{k=1}^{n}\frac{\delta\tL^u\sigma_k}{\sigma_k}(f^{k-n}x)\Phi\right).
\end{split} \]
Notice that at $\gamma=0$, we have $\sigma_k = \sigma$, and $\rho$ is $f$-invariant.
Let $m=n-k$, 
\[ \begin{split}
\delta\rho_n(\Phi) 
=
\sum_{k=1}^{n} \rho\left(\frac{\delta\tL^u\sigma}{\sigma} \circ f^{k-n} \, \Phi\right)
=
\sum_{m=0}^{n-1} \rho\left(\Phi\circ f^m\frac{\delta\tL^u\sigma}{\sigma} \right).
\end{split} \]
Finally, let $n\rightarrow\infty$, we formally get  
\[ \begin{split}
\delta \tilde \rho (\Phi, X^u)
= \sum_{m\geq 0}\rho\left(\Phi\circ f^m\frac{\delta\tL^u\sigma}{\sigma} \right).
\end{split} \]
This is equivalent to the expression in \cref{t:far}, except for that here we are using the stable/unstable decomposition instead of the shadowing/unstable decomposition of the linear response.

\section{Conclusions}

The phase space is typically high dimensional, so efficient computations demand sampling by an orbit rather than approximating high-dimensional measures by finite-elements.
In this paper we solve this problem for the more difficult part, the unstable perturbation operator of a physical measure.
It was well known that the physical measure can be sampled by an orbit; now, with our results, we know that the unstable derivative of the transfer operator and hence the linear response can also be sampled on a orbit by $2u$ recursive relations. 
This cost is perhaps optimal, since we need at least $u$ many modes to capture all the unstable perturbative behaviors of a chaotic system.

\section*{Acknowledgements}

Both authors thank Yi Shi for helpful discussions.
Angxiu Ni very grateful to Stefano Galatolo, Wael Bahsoun, Gary Froyland, and Caroline Wormell for very helpful discussions.
Angxiu Ni is partially supported by the China Postdoctoral Science Foundation 2021TQ0016 and the International Postdoctoral Exchange Fellowship Program YJ20210018.
This work is partially done during Angxiu Ni's postdoctoral period at Peking University and during his visit to Mark Pollicott at the University of Warwick.

\section*{Data Availability Statements}

This manuscript has no associated data.
The paper detailing the algorithm for \cref{t:vdiv} and \cref{t:far} is \cite{far}, and the code is at \url{https://github.com/niangxiu/far}.

\begin{appendix}

\section{A very rough cost estimation of finite-element method for approximating high-dimensional measures}
\hfill\vspace{0.1in}
\label{s:costEst}

When the measure is singular, $\delta \tL h$ has infinite sup norm.
Although this is a well-defined mathematical objects in suitable Banach spaces, computers can not process infinite sup norm.
Currently, the main numerical practice for computing derivative operators is to first approximate the measure by finite-elements  \cite{Keane1998,Froyland2007,Liverani2001,Ding1994,Pollicott2000,Galatolo2014,Galatolo2014a,Wormell2019,Crimmins2020,Antown2021,Wormell2019a,Froyland2013a,Zhang2020,Pollicott2016}, then compute the derivative operator.
Here the finite-elements are in the vague sense, which includes both finite-element, finite-difference, kernel method, Markov partition, Markov approximation.
This approximation approach allows us to ignore the singularities and subtle structures of measures; however, the cost is still affected.

Computing the entire derivative operator was not numerically realized for discrete-time systems with dimensions larger than 1: for example, Bahsoun, Galatolo, Nisoli, and Niu did computations on 1-dimensional expanding maps \cite{Bahsoun2018}.
For continuous-time systems, Gutiérrez and Lucarini numerically computed the derivative operator for a continuous-time 3-dimensional system \cite{SantosGutierrez2020},
There are two difficulties, the easier one is the lack of convenient formulas, which we solved via theorem~\ref{t:div}.
The more essential difficulty is that the finite-element method is cursed by the dimension of the dynamical system.

Galatolo and Nisoli gave a rigorous posterior error bound for the finite-element method, where some quantities in the bound are designated to be computed by numerical simulations \cite{Galatolo2014a}.
That bound, though precise, does not give the cost-error relation and how it depends on dimensions.

This section gives an a priori cost-error estimation on a simple singular measure approximated by zero-order isotropic finite elements.
A more general and precise estimation is more difficult, but should not change the qualitative conclusion.
That is, the cost is `cursed by dimensionality', or it increases exponentially fast with respect to the dimension of the attractor.
For physical or engineering systems, this cost is too high.

Consider the example where the singular measure is uniformly distributed on the $a$-dimensional attractor, $\{0\}^{M-a}\times \T^a$, where $\T:=[-0.5,0.5]$.
We use the zeroth order finite-elements in the $M$-dimensional cubes of length $b$ on each side.
The density $h$ is a distribution; and we still formally denote the SRB measure as the integration of $h$ with respect to the Lebesgue measure.
Let $h' $ be the finite-element approximation of $h$, so
\[ \begin{split}
  h' (x) = \begin{cases}
  b^{-(M-a)} \quad \textnormal{for} \quad 
  |x^1|, \ldots, |x^{M-a}| \le b/2; \;
  |x^{M-a+1}|, \ldots, |x^{M}| \le 0.5; \\
  0 \quad \textnormal{otherwise}.
\end{cases} \end{split} \]
For the smooth objective function, $\Phi$, we assume that for any unit tangent vector $Y$, the second order derivative
$Y^2(\Phi)(x):=Y(Y(\Phi)) (x) \le 1$,
where we use geometer's notation that $Y(\Phi)=\grad \Phi \cdot Y$.

The approximation error $E$ caused by using $h' $ instead of $h$ is
\[ \begin{split}
  &E :=  \int \Phi  h' dx - \int \Phi h dx\\
  &= \int_{\T^{a}} \int_{\T^{M-a}} \Phi h' dx^{1\sim M-a} dx^{M-a+1\sim M}
  - \int_{\T^{a}} \Phi(0, x^{M-a+1},\ldots, x^M) dx^{M-a+1\sim M} \\
  &=\int_{\T^{a}} E(x^{M-a+1\sim M}) \, d x^{M-a+1\sim M},
\end{split} \]
where $dx^{1\sim M-a} = dx^1\ldots dx^{M-a}$, and
\[ \begin{split}
E(x^{M-a+1\sim M}) 
:= \int_{\T^{M-a}} \phi h' dy - \phi(0) 
= \int_{\T^{M-a}} (\phi(y)-\phi(0)) h' dy ,
\end{split} \]
where $\phi (y) := \Phi (y, x^{M-a+1},\ldots, x^M)$, $y\in\T^{M-a}$.
Then by Taylor expansion, there is $\xi(y)$ such that
\[ \begin{split}
E(x^{M-a+1\sim M}) 
= \int_{\T^{M-a}} \left(Y(\phi) (0) |y| + Y^2(\phi)(\xi(y)) \frac{|y|^2}2 \right) h' dy.
\end{split} \]
Here $Y:=y/|y|$.
The first term is zero due to symmetry, hence by assumptions
\[ \begin{split}
|E(x^{M-a+1\sim M})|
\le \int_{\T^{M-a}}  |y|^2 h' dy
\\
= b^{-(M-a)} \int_{[-0.5b, 0.5b]^{M-a}}  (y^1)^2 +\ldots + (y^{M-a})^2 dy^1\ldots dy^{M-a}
\\
=  b^{-(M-a)} (M-a) \frac13 b^3 b^{M-a-1}
= \frac13 (M-a) b^2.
\end{split} \]
Hence,
\[ \begin{split}
  |E| 
  \le \int_{\T^{a}} \left|E(x^{M-a+1\sim M}) \right| \, d x^{M-a+1\sim M}
  \le (M-a) b^2.
\end{split} \]
For more general cases, there should be another error due to approximation within the attractor, but here we neglect it.

It is nontrivial to achieve optimal mesh adaptation in higher dimensions.
For now we assume optimal mesh, then we can restrict our computation to the attractor, and the cost is at least propotional to the number of cells in the mesh, so
\[ \begin{split}
S\sim \left(\frac1b\right)^{a} \sim \left( \frac{M-a}{E} \right)^{\frac a2} .
\end{split} \]
When the attractor dimension $a$ is higher than 4, this cost is much larger than the cost of our equivariant divergence formula estimated in \cref{e:hua}.
On the other hand, if the finite-elements are all globally supported, such as the Fourier basis,
or if the optimal implementation is not achieved, the cost can be as high as $O(b^{-M})$.

In higher dimensions, it is expensive to approximate the entire attractor by finite-elements.
Hence, it is also expensive to compute $\delta \tL h$ via finite-elements.

\section{Equivalence between the equivariant divergence formula and the fast tangent response formula}
\label{a:douyin}

This section shows that the fast (tangent) formula in our previous paper \cite{fr} is equivalent to the fast adjoint formula for the linear response in \cref{t:far} of this paper.
Both formulas can be either proved directly, or by first proving the other formula and then proving the equivalence between the two.
However, proving \cref{t:far} from the indirect approach is less intuitive;
moreover, if we want to start from \cref{t:fastformula}, prove \cref{t:far}, then prove \cref{t:vdiv}, then we can only prove \cref{t:vdiv} for $\rho$-almost everywhere.
Besides giving another proof of the main results in this paper, the equivalence also verifies the fast formula in our previous paper, which runs only forward along an orbit.

The main difference, in terms of utility, between the adjoint and tangent formulas is that the adjoint is more suitable for cases with many parameters, whereas the tangent is suitable for cases with a few parameters.
Because if we want to compute linear responses with respect to many perturbations, each controlled by a separate parameter $\gamma$, then we would have to compute our formula for many different $X$'s.
The main term in the adjoint formula, $\cS(\div^vf_*)$, does not depend on $X$, so this main term needs not be recomputed; the case with multiple parameters, solved by the adjoint algorithm, is tested in another paper \cite{far}.
In contrast, the main term in the fast tangent formula needs to be recomputed for each $X$, so the marginal cost is larger for a new parameter.
On the other hand, when we have only one parameter, then the tangent algorithm is faster, since it involves only the pushforward of vectors, whose computation is faster than the pullback of covectors in terms of clock time, even though the the number of flops (float
point operations) is the same.

The main idea to prove the equivalence is `adjointing' the fast formula of the unstable contribution. 
More specifically, we shall expand the unstable divergence, move major computations away from $X$ and $\phi_W$, and obtain an expansion formula for an adjoint operator.
Then we seek a neat characterization of the expansion formula, and we prove \cref{t:far} by further using the adjoint shadowing lemma.

\subsection{Fast formula for the unstable contribution}
\hfill\vspace{0.1in}
\label{s:UC}

Recall that the unstable contribution is 
\begin{equation} \begin{split} \label{e:uc}
  U.C.^W = \rho \left(\phi_W  \diverg_\sigma^u  X^u\right),
  \quad \textnormal{where} \quad 
  \phi_W := \sum_{m=-W}^W (\Phi \circ f^m-\rho(\Phi).
\end{split} \end{equation}
Here $\diverg_\sigma^u$ is the submanifold divergence on the unstable manifold under the conditional SRB measure.
The norm of this integrand is $O(\sqrt W)$, much smaller than the ensemble formula.
Note that the directional derivatives of $X^u$ are distributions.
We gave a fast formula for the unstable divergence.
It involves only $u$ many second-order tangent equations on one sample orbit, which runs forwardly in time.

\begin{theorem} [fast formula for unstable contribution \cite{fr}] \label{t:fastformula}
  Let $\{x_n:=f^nx_0\}_{n\ge0}$ be a orbit on the attractor, then for almost all $x_0$ according to the SRB measure $\rho$, for any $r_0\in \cD^u(x_0)$, 
  \[ \begin{split}
    U.C.^W 
    = \lim_{N \rightarrow \infty} \frac 1 {N} \sum_{n=0}^{N-1}
      \ip{r_n ,
      e_{n} },
    \quad\textnormal{where}\quad
    r_{n+1}= \tilde \beta P^\perp  r_n \in \cD^u(x_{n+1} ) .
  \end{split} \]
\end{theorem}
\noindent{}Here $e=e_1\wcw e_u$ is the unit $u$-dimensional cube spanned by unstable vectors, 
and $\span \{e_i\}_{i=1}^u=V^u$.
Here $\ip{\cdot, \cdot}$ is the inner product between $u$-vectors;
\begin{equation} \begin{split} \label{e:Du}
  \cD^u := \{r= \sum_i e_1\wedge\cdots\wedge r_i \wedge\cdots\wedge e_u:\;
    r_i\in T_x\cM, e_i\in V^u\}.
\end{split} \end{equation}
is the space of derivatives of unstable cubes; 
note that the definition of $\cD^u$ does not depend on the selection of $e_i$'s, so long as they span $V^u$.
For any $r\in \cD^u$,
\begin{equation} \begin{split} \label{e:nn1}
  P^\perp r:= \sum_i e_1 \wcw  P^\perp r_i \wcw e_u,
\end{split} \end{equation}
where the second $P^\perp$ orthogonally projects a vector to the subspace perpendicular to $V^u$.
Here $\tilde \beta$ is the renormalized second-order tangent equation governing the propagation of derivatives of cubes,
\begin{equation} \begin{split} \label{e:hai}
  r_{n+1} = \tilde \beta P^\perp r_n 
  = \frac{f_* P^\perp } J r_n 
  + q_{n+1},
  \quad \textnormal{where} \quad 
  q_{n+1} = \frac{(\nabla_{\tv_n} f_*) } J e_{n} + (\phi_W \nabla_{e} X)_{n+1}.
\end{split} \end{equation}
Here the Jacobian determinant $J:= |f_*e_n|$ 
when applied to quantities at $x_n$.
$\tv:=S(\phi_W X)$ is the shadowing vector of $\phi_W X$, which is the only bounded solution to
\begin{equation}
  \tv_{n+1} = f_*\tv_n + (\phi_W X)_n.
\end{equation}

The fast response algorithm based on the fast formula was demonstrated on a 21-dimensional system with a 20-dimensional unstable subspace.
For the same accuracy, fast response is orders of magnitude faster than ensemble and operator algorithms; it is even faster than finite difference \cite{fr}.

\subsection{Expansion formulas of unstable contribution}
\hfill\vspace{0.1in}
\label{s:expand}

\begin{lemma}
\label{l:expandUC}
The orbit-wise average expression in \cref{t:fastformula} converges $\rho$-almost surely according to $x_0$ to
\[ \begin{split}
\lim_{N \rightarrow \infty} \frac 1 {N} \sum_{n=0}^{N-1} \ip{r_n , e_{n} }
\overset{a.e}{=} \rho \left( e^b Tq \right) 
= \rho\left((\cT e^b) q\right),
\end{split} \]
where
\[ \begin{split}
  q := \frac{\nabla_{e_{-1}} f_* } J \tv_{-1} + \phi_W \nabla_{e} X,
  \quad
  e^b(\cdot):=\ip{e,\cdot},
  \\
  T(\cdot) :=\sum_{k\ge0} \left( \frac{f_* P^\perp} {J} \right) ^ {k} (\cdot)_{-k},
  \quad
  \cT e^b (\cdot) := \sum_{k\ge0} \ip{ \left( \frac{f_* P^\perp} {J} \right) ^ {k} (\cdot), e_{k} }.
\end{split} \]
Here $e_{-1} = e\circ f^{-1}$, $v_{-1}=v\circ f^{-1}$.
  Note that this lemma does not assume that the expressions are related to linear response.
\end{lemma}

\begin{remark*}
(1)
Here $q$ is in the space of derivative-like $u$-vectors $\cD^u$, 
$e^b$ is in the dual space $\cD^{u*}$;
$T$ is a linear operator on $L^2(\rho, \cD^u)$,
its adjoint operator is $\cT$, which is an operator on $L^2(\rho, \cD^{u*})$.
(2) 
This is not the full expansion, since $\tv$ can be further expanded.
But there is no need to expand $\tv$, since the adjoint shadowing lemma has told us very well how to deal with it.
Also note that the expansion formula of $p$ in \cite{fr} is also not the full expansion, and is different from our current one.
\end{remark*}

\begin{proof}
We can use the linear superposition law to write out the explicit solution of the inductive relation of $r$ in \cref{e:hai},
\[ \begin{split}
  r_{n} 
  = \left( \frac{ f_*P^\perp} {J} \right) ^ {n} r_0
  + \sum_{1\le k\le n} \left( \frac{ f_* P^\perp} {J} \right) ^ {n-k} q_{k}.
\end{split} \]

For any $x_0\in K$ and any $r_0\in \cD^u(x_0)$ and $q_{k}\in \cD^u (x_k)$
\[ \begin{split}
\left| \left(\frac{ f_* P^\perp} {J} \right) ^n r_0 \right| \rightarrow 0,
\quad\quad
\left| \left( \frac{ f_* P^\perp} {J} \right) ^ {n-k} q_{k}\right|\rightarrow 0
\end{split} \]
exponentially fast as $n\rightarrow \infty$.
To see this, notice that the invariance of the stable and unstable subspace indicates
\[ \begin{split}
(f_* P^\perp)^n
= f_* P^\perp f_*^{n-1} P^\perp
= f_* P^\perp f_*^{n-1} 
= f_* P^\perp P^s f_*^{n-1} 
= f_* P^\perp f_*^{n-1} P^s,
\end{split} \]
where $P^s$ is defined similar to \cref{e:nn1}. 
Since $P^s$ keeps only the stable part, which decays during pushfowards, $|(\frac{ f_* } {J})^{n-1} P^s r_0|$ decays exponentially fast as $n\rightarrow\infty$.

Hence, $r_n - Tq(x_n)\rightarrow 0$ for any $x_0$ and $r_0$, so
\[ \begin{split}
\lim_{N \rightarrow \infty} \frac 1 {N} \sum_{n=0}^{N-1} \ip{r_n , e_{n} }
= \lim_{N \rightarrow \infty} \frac 1 {N} \sum_{n=0}^{N-1} \ip{Tq(x_n) , e(x_n) }.
\end{split} \]
Note that $Tq$ is a continuous function on $K$ since it is the sum of a uniformly convergent series of function; in fact, it is even Holder continuous, by the same argument given in the appendix of \cite{Ni_asl}.
Hence, by the Birkhoff's ergodic theorem, for $\rho$-almost all $x_0$, the sum to the integration by $\rho$, that is
\[ \begin{split}
\lim_{N \rightarrow \infty} \frac 1 {N} \sum_{n=0}^{N-1} \ip{Tq(x_n) , e(x_n) }
\overset{a.e}{=} \rho \left(\ip{Tq , e }\right)
= \rho \left(e^b Tq  \right).
\end{split} \]

The second equality in the lemma is due to the duality between pushfoward on vectors and pullback on covectors, and the invariance of SRB measures.
\end{proof}

\subsection{Characterizing \texorpdfstring{$\cT e^b$}{Te} by unstable co-cube}
\hfill\vspace{0.1in}

\begin{lemma}  \label{l:cocube}
$\cT e^b(r) = \eps(r)$, $\forall r\in \cD^u$.
\end{lemma}

\begin{proof}
Since $P^\perp f_* P^\perp = P^\perp f_*$,
$P^\parallel:=I_d - P^\perp$ orthogonally projects to the span of $e$, we see that for any $r\in \cD^u$,
\[ \begin{split}
  \ip{ \left( \frac{f_* P^\perp} {J} \right) ^ {k} r, e_{k} } 
  = \ip{ \frac{f_* P^\perp f_*^{k-1}} {J^k}  r, e_{k} } 
  = \ip{ \frac{f_*^{k}} {J^k}  r, e_{k} } 
  - \ip{ \frac{f_* P^\parallel f_*^{k-1}} {J^k}  r, e_{k} } \\
  = \ip{ \frac{f_*^{k}} {J^k}  r, e_{k} } 
  - \ip{ \frac{P^\parallel f_*^{k-1}} {J^{k-1}}  r, e_{k-1} } 
  = \ip{ \frac{f_*^{k}} {J^k}  r, e_{k} } 
  - \ip{ \frac{f_*^{k-1}} {J^{k-1}}  r, e_{k-1} }. \\
\end{split} \]
The last equality is because $\ip{P^\parallel \cdot, e}= \ip{\cdot, e}$.
Hence, 
\[ \begin{split}
  \cT e^b(r)
  = \ip{f, e_0} + \lim_{N\rightarrow\infty}
  \sum_{k\ge1}^N \ip{ \frac{f_*^{k}} {J^k}  r, e_{k} } 
  - \ip{ \frac{f_*^{k-1}} {J^{k-1}}  r, e_{k-1} } 
  = \lim_{N\rightarrow\infty} \ip{ \frac{f_*^{N}} {J^N}  r, e_{N} } .
\end{split} \]
The lemma is proved once we show that this equals $\eps(r)$.
Intuitively, this is because the stable parts decay while the unstable parts grow. 
A more careful proof of this statement is in appendix~\ref{s:zhuang}.
\end{proof}

Finally, we can show that the fast (forward) formula is equivalent to the fast adjoint formula for unstable contributions.
Note that here we are directly showing that the two expressions are equivalent, without knowing that these expressions are in fact for the unstable contribution of the linear response.

\begin{theorem} \label{t:equi}
The expression in \cref{t:fastformula} is equivalent to the expression in \cref{t:far}.
That is,
\[ \begin{split}
\lim_{N \rightarrow \infty} \frac 1 {N} \sum_{n=0}^{N-1} \ip{r_n , e_{n} }
\overset{a.e}{=} \rho\left((\cT e^b) q\right)
= \rho\left(\phi_W(\cS(\div^v f_*) X + \div^v X )\right).
\end{split} \]
\end{theorem}

\begin{proof}
The first equality is due to \cref{l:expandUC}.
For the second equality, substitute lemma~\ref{l:cocube} and the definition of $q$, we get
\[ \begin{split}
\rho\left((\cT e^b) q\right)
= \rho\left(\eps \frac{\nabla_{e_{-1}} f_* } J \tv_{-1}
  + \eps \phi_W \nabla_{e} X \right)
= \rho\left((\div^v f_*) \tv + \phi_W \div^v X \right) .
\end{split} \]
Recall that $\tv:=S(\phi_W X)$ is the shadowing vector for $\phi_W X$; hence, by the duality between $S$ and $\cS$, which is the first characterization of $\cS$ listed in \cref{s:recurs}, we have
\begin{equation} \begin{split} \label{e:lishi}
\rho\left((\cT e^b) q\right)
= \rho\left(\phi_W(\cS(\div^v f_*) X + \div^v X )\right) .
\end{split} \end{equation}
\end{proof}

Furthermore, we can prove \cref{t:vdiv} for $\rho$-almost everywhere, starting from the fast (forward) formula.
This follows from comparing \cref{e:lishi} with \cref{e:uc}, since the equalities hold for any smooth $\phi_W$, \cref{t:vdiv} must hold $\rho$-almost everywhere.
This is slightly weaker than \cref{t:vdiv} proved in the main body of the paper, which holds on the entire hyperbolic attractor.

\section{relative decay of pushing forward \texorpdfstring{$\cD^u$}{Du}}
\label{s:zhuang}

We prove a theorem used in both proofs of the $v$-divergence formula.
This theorem is essentially the dominated splitting on the $u$-dimensional Grassmannian: it says $e$ grows exponentially faster than all other $u$-vectors.

\begin{theorem} [relative decay of pushing forward $\cD^u$] \label{l:alice}
For any $r$ in the space $\cD^u$ defined in \cref{e:Du},
\[ \begin{split}
  \left|\ip{ \frac{f_*^{N}r} {|f_*^{N} e |}, e_{N} } - \eps (r) \right|
  \le C u \lambda^{2N} |r|.
\end{split} \]
Here $|e|=|e_N|=1$.
This difference goes to zero as $ N\rightarrow \infty$.
\end{theorem}

\begin{proof}
Assume for convenience that $\{ e_i \}_{ i=1 }^u$ is an orthonormal basis for $V^u$, and denote $e:=e_1\wcw e_u$.
Note that $\cD^u$ is a subspace of the Grassmannian $Gr(u,T\cM)$, and it admits the decomposition
\[ \begin{split}
  \cD^u = \wedge^uV^u \bigoplus V^s\wedge^{u-1}V^u.
\end{split} \]
Hence we can rewrite $r$ on the orthonormal basis $\{ e_i \}_{ i=1 }^u$, as
\[ \begin{split}
r = a e + \sum_i e_1\wcw r_i\wcw e_u,
\quad \textnormal{where} \quad a\in\R, r_i\in V^s.
\end{split} \]
See the appendix of \cite{fr} for a change of basis formula.

Since $\eps (e) = 1 = \ip{e_N,e_N} = \ip{f_*^Ne/|f_*^Ne|, e_N}$, the first term in $r$ cancels with $\eps(r)$.
Also notice that since $r_i\in V^s$, $\eps(e_1\wcw r_i\wcw e_u)=0$, so
\[ \begin{split}
  \ip{ \frac{f_*^{N}r} {|f_*^{N} e |}, e_{N} } - \eps (r)
  = \sum_{i=1}^u \frac 1 {|f_*^{N} e |} 
  \ip{ f_*^{N} (e_1\wcw r_i\wcw e_u), e_{N} }.
\end{split} \]
Intuitively, because that $r_i\in V^s$ decays exponentially fast, the above formula decays to zero as $N\rightarrow\infty$.
But we should be more careful when proving it for $u$-vectors.
We first prove the following lemma.

%We claim that
%\[ \begin{split}
%\cT e ^b(r)=
%\lim_{N\rightarrow\infty} \ip{ \frac{f_*^{N}r} {|f_*^{N} e |}, e_{N} } =
%\begin{cases}
  %1, \quad\text{if } r =  e  , \\
  %0, \quad\text{if } r = \sum_i e_1\wcw r_i\wcw e_u, r_i\in V^s.
%\end{cases}
%\end{split} \]
%If the claim is true, $\cT e ^b=\eps$ for all basis of $\cD^u$,
%hence, the equality also holds for all $r\in \cD^u$,
%finishing the proof of the lemma.
%The first case is straightforward, and we only need to prove the second case.

\begin{lemma} \label{l:mei}
In this lemma, let $C$ be the constant used in the definition of hyperbolicity.
If $\{ e_i \}_{ i=1 }^u$ is an orthogonal basis for $V^u$, denote $ e: =e_1\wcw e_u$, then
\[ \begin{split}
  C\lambda^n |f_*^n e | \ge|f_*^n(e_2\wcw e_u)| |e_1|, 
  \quad \textnormal{} \quad 
  \forall n\ge0.
\end{split} \]
\end{lemma}

\begin{proof}
Assume for contradiction that the lemma is false, that is, for some $n$, 
\[ \begin{split}
  C\lambda^n |f_*^n e | < |f_*^n(e_2\wcw e_u)| |e_1|.
\end{split} \]
We can find $e_1'\in V^u$, such that
$f^n_* e_1'\perp \spanof\{f_*^n e_2, \cdots, f_*^n e_u\}$,
and $e_1'\wedge e_2\wcw e_u =  e $.
As a result, $|e_1'|\ge|e_1|$. Hence, by our assumption,
\[ \begin{split}
C\lambda^n |f_*^n(e_2\wcw e_u)| |f^n_* e_1'|
=   C\lambda^n |f_*^n e | < |f_*^n(e_2\wcw e_u)| |e_1|
\end{split} \]
\[ \begin{split}
  \Rightarrow C\lambda^n |f^n_* e_1'| < |e_1| \le |e_1'|.
\end{split} \]
Denote $w:=f^n_* e_1'$, then  $w\in V^u$, but $C\lambda^n |w| < |f_*^{-n}w|$,
contradicting our hyperbolicity assumption.
\end{proof}

With lemma~\ref{l:mei}, also note that $|e_1|=1$, and that  $|f_*^{N} (r_1\wedge e_2\wcw e_u)|<|f_*^{N} r_1|\, |f_*^N (e_2\wcw e_u)|$,
\[ \begin{split}
  \frac{|f_*^{N} (r_1\wedge e_2\wcw e_u)|} {|f_*^{N}  e |} 
  \le C \lambda^N \frac{|f_*^{N} r_1|\, |f_*^N (e_2\wcw e_u)|}
  {|e_1| \, |f_*^N (e_2\wcw e_u)|} \\
  = C \lambda^N \frac{|f_*^{N} r_1|} {|e_1|} 
  \le C \lambda^{2N} \frac{|r_1|} {|e_1|} 
  = C \lambda^{2N} |r_1|.
\end{split} \]
Note that here different constant $C$'s may take different values in each expression.

Similarly, we run the above arguments for any $r_i$, to get
\[ \begin{split}
  \left|\ip{ \frac{f_*^{N}r} {|f_*^{N} e |}, e_{N} } - \eps (r) \right|
  = \sum_{i=1}^u \frac 1 {|f_*^{N} e |} 
  \ip{ f_*^{N} (e_1\wcw r_i\wcw e_u), e_{N} }\\
  \le \sum_{i=1}^u \frac {|f_*^{N} (e_1\wcw r_i\wcw e_u)|} {|f_*^{N} e |} 
  \le C \lambda^{2N} \sum_{i=1}^u |r_i|.
\end{split} \]
In the above we used the fact that $|\cdot|$ is the norm induced by the inner-product on $u$-vectors.

We still need to bound $|r_i|$ by $|r|$.
Since the angle between $V^s$ and $V^u$ is bounded from below by a positive number, we have $|r_i|\le C |r_i^\perp|$, where $r_i^\perp := P^\perp r_i$.
Also, we can decompose $r$  into orthogonal $u$-vectors
\[ \begin{split}
r = a' e + \sum_i e_1\wcw r_i^\perp \wcw e_u,
\quad \textnormal{where} \quad a'\in\R, \; r_i^\perp\perp V^u.
\end{split} \]
Hence $|r|^2=|a'e|^2+\sum|r_i^\perp|^2$, so
\[ \begin{split}
  |r| \ge |r_i^\perp|\ge C |r_i|.
\end{split} \]
Summarizing, we have
\[ \begin{split}
  \left|\ip{ \frac{f_*^{N}r} {|f_*^{N} e |}, e_{N} } - \eps (r) \right|
  \le C \lambda^{2N} \sum_{i=1}^u |r_i|
  \le C u \lambda^{2N} |r|.
\end{split} \]
\end{proof}

\end{appendix}

\bibliographystyle{abbrv}
{\footnotesize\bibliography{library}}

\end{document}